\documentclass{amsart}

\usepackage{amsmath,amssymb,amsthm}
\usepackage[all]{xy}

\newtheorem{theorem}{Theorem}[section]
\newtheorem{corollary}[theorem]{Corollary}
\newtheorem{proposition}[theorem]{Proposition}
\newtheorem{lemma}[theorem]{Lemma}
\theoremstyle{definition}
\newtheorem{definition}[theorem]{Definition}
\theoremstyle{remark}
\newtheorem{remark}[theorem]{Remark}
\newtheorem{example}[theorem]{Example}

\DeclareMathOperator{\id}{id}
\DeclareMathOperator{\pr}{pr}
\DeclareMathOperator{\Diff}{Diff}
\DeclareMathOperator{\SO}{SO}
\DeclareMathOperator{\GL}{GL}

\DeclareMathOperator{\Aut}{Aut}
\DeclareMathOperator{\amb}{amb}
\DeclareMathOperator{\Ham}{Ham}
\DeclareMathOperator{\std}{std}
\DeclareMathOperator{\Image}{Im}
\DeclareMathOperator{\End}{End}
\DeclareMathOperator{\Vect}{Vect}
\DeclareMathOperator{\Real}{Re}

\begin{document}

\title{Deformation of Sasakian metrics}
\author{Hiraku Nozawa}
\address{Unit\'{e} de Math\'{e}matiques Pures et Applique\'{e}s\\
\'{E}cole Normale Sup\'{e}rieure de Lyon (site Sciences)
\\
46, all\'{e}e d'Italie 69364 Lyon Cedex 07, France}
\email{nozawahiraku@06.alumni.u-tokyo.ac.jp}
\thanks{The author was partially supported by Grant-in-Aid for JSPS Fellows (19-4609). The author was partially supported by Postdoctoral Fellowship of French government (662014L)}
\keywords{Deformation theory, Sasakian metrics, transversely holomorphic foliations}
\subjclass[2000]{Primary: 32G07; Secondary: 53C25}

\begin{abstract}
Deformations of the Reeb flow of a Sasakian manifold as transversely K\"{a}hler flows may not admit compatible Sasakian metrics. We show that the triviality of the~$(0,2)$-component of the basic Euler class characterizes the existence of compatible Sasakian metrics for given small deformations of the Reeb flow as transversely holomorphic Riemannian flows. We also prove a Kodaira-Akizuki-Nakano type vanishing theorem for basic Dolbeault cohomology of homologically orientable transversely K\"{a}hler foliations. As a consequence of these results, we show that any small deformations of the Reeb flow of a positive Sasakian manifold admit compatible Sasakian metrics.
\end{abstract}
\maketitle

\tableofcontents

\section{Introduction}
\addtocontents{toc}{\protect\setcounter{tocdepth}{1}}

\subsection{Characterization of the stability of Sasakian metrics}
The Reeb flow of the contact form $\eta$ of a Sasakian manifold has a transverse K\"{a}hler structure, which is one of important underlying structures of Sasakian manifolds as well as $CR$-structures. In this paper, we investigate under which conditions given families of deformation of the Reeb flow of $\eta$ as transversely holomorphic Riemannian flows have compatible Sasakian metrics. Although any small deformations of the complex structure of a compact K\"{a}hler manifold admit compatible K\"{a}hler metrics by a theorem of Kodaira and Spencer~\cite{Kodaira Spencer}, there exists a Sasakian manifold whose Reeb flow can be deformed so that it does not admit a compatible Sasakian metric. We present an example of such a Sasakian metric on a circle bundle over a complex torus in Section~\ref{Section : Example}. Seeing the example we note that the nontriviality of the~$(0,2)$-component of the basic Euler class of flows gives an obstruction for the existence of compatible Sasakian metrics. Our main result shows that this obstruction is unique, that is, small deformations of the Reeb flow admit compatible Sasakian metrics, if the basic Euler class is of degree~$(1,1)$. Note that Boyer and Galicki discussed a similar problem in Section~8.2 of~\cite{Boyer Galicki}. The product of $S^{1}$ with a Sasakian manifold admits geometric structures called a Vaisman manifold or a locally conformal K\"{a}hler manifold (see Ornea and Verbitsky~\cite{Ornea Verbitsky 2}). Belgun~\cite{Belgun} showed that both of structures are not stable under small deformations of complex structures. Note that Ornea and Verbitsky~\cite{Ornea Verbitsky} discovered the stability of another class of geometric structure under small deformations of complex structures, which they call locally conformal K\"{a}hler manifolds with potential.

A~$1$-dimensional foliation is called a flow in this article in accordance with references. Our main result is stated as follows: Let $V$ be an open neighborhood of~$0$ in $\mathbb{R}^{\ell}$ and $\{(\mathcal{F}^{t},J^{t},g^{t}_{\nu})\}_{t \in V}$ be a smooth family of transversely holomorphic Riemannian flows on a closed manifold~$M$. Assume that $(\mathcal{F}^{0},J^{0},g^{0}_{\nu})$ is the underlying transversely K\"{a}hler structure of the Reeb flow of a Sasakian metric $(g,\eta)$. 
\begin{theorem}\label{Theorem : Stability}
If the basic Euler class of $(\mathcal{F}^{t},J^{t})$ is of degree~$(1,1)$ for every $t$ in $V$, then there exists an open neighborhood $V_{1}$ of~$0$ in $V$ and a smooth family of Sasakian metrics $\{({g}^{t},\eta^{t})\}_{t \in V_{1}}$ on~$M$ such that the underlying transversely holomorphic flow of the Reeb flow of $({g}^{t},\eta^{t})$ is $(\mathcal{F}^{t},J^{t})$ for every $t$ in $V_{1}$ and $({g}^{0},\eta^{0})=({g},\eta)$.
\end{theorem}
\noindent For the existence of a smooth family of compatible Sasakian metrics, it is necessary that the basic Euler class of $(\mathcal{F}^{t},J^{t})$ is of degree~$(1,1)$ (see Lemma~\ref{Lemma : Sasakian case}). In this sense, Theorem~\ref{Theorem : Stability} characterizes the stability of the Sasakian metric $(g,\eta)$.

The difficulty to prove Theorem~\ref{Theorem : Stability} comes from the noncontinuous change of the complex of basic differential forms of foliations. The families of partial differential equation to solve to prove Theorem~\ref{Theorem : Stability} is discontinuous with respect to the parameter $t$ because of this discontinuous change of basic de Rham complex. To avoid this difficulty, we will apply the invariance of geometrically tautness of Riemannian flows obtained in~\cite{Nozawa 2}. This result allows us to convert problems on families of basic de Rham complexes to problems on families the de Rham complexes, where we can apply the Kodaira-Spencer theory~\cite{Kodaira Spencer} on smooth families of self-adjoint strongly elliptic differential operators.

\subsection{Vanishing theorem for basic cohomology of positive Sasakian manifolds}

We show
\begin{theorem}\label{Theorem : Vanishing}
Let $(M,\mathcal{F})$ be the underlying transversely K\"{a}hler flow of a positive Sasakian manifold. Then we have
\begin{equation*}
H^{0,q}_{b}(M/\mathcal{F}) = 0
\end{equation*}
for $q > 0$ and 
\begin{equation*}
H^{p,0}_{b}(M/\mathcal{F}) = 0
\end{equation*}
for $p > 0$.
\end{theorem}
\noindent Theorem~\ref{Theorem : Vanishing} is a generalization of a part of Proposition~2.4 of Boyer, Galicki and Nakamaye~\cite{Boyer Galicki Nakamaye}, where they assume that Sasakian manifolds are quasi-regular, that is, every leaf of $\mathcal{F}$ is closed. In this case, the leaf space $M/\mathcal{F}$ is a complex orbifold. Thus they deduced Theorem~\ref{Theorem : Vanishing} from Kodaira-Baily vanishing theorem for complex orbifold due to Baily~\cite{Baily}.

To show Theorem~\ref{Theorem : Vanishing} for general positive Sasakian manifolds, there were two problems: The first problem is that the Kodaira-Akizuki-Nakano vanishing theorem was not known for general Sasakian manifolds whose Reeb flow has a non-closed orbits. The second problem was that $H^{p,q}_{b}(M/\mathcal{F}, E)$ may not be isomorphic to the sheaf cohomology $H^{q}(M, \Omega^{p}_{b} \otimes E)$, where $\Omega^{p}_{b}$ is the sheaf of basic holomorphic $p$-forms on $(M,\mathcal{F})$ as we will mention in Remark~\ref{Remark : sheaf}. Because of these difficulty, the classical argument to show Theorem~\ref{Theorem : Vanishing} for Fano manifolds did not work for basic cohomology of positive Sasakian manifolds. The second problem was pointed out by Boyer, Galicki and Nakamaye before Proposition~2.4 of~\cite{Boyer Galicki Nakamaye}. We show Theorem~\ref{Theorem : Vanishing} by solving these two problems: We will show that Kodaira-Akizuki-Nakano vanishing theorem is true for general homologically orientable transversely K\"{a}hler foliations as mentioned in the next section. We will solve the second problem by a simple observation on isomorphisms between the spaces of harmonic forms (see Section~\ref{Section : pos}).

\subsection{Kodaira-Akizuki-Nakano vanishing theorem for transversely K\"{a}hler foliations}
Let $(M,\mathcal{F})$ be a closed manifold with a homologically orientable transversely K\"{a}hler foliation of complex codimension $n$. Let $E$ be an $\mathcal{F}$-fibered Hermitian holomorphic line bundle over $(M,\mathcal{F})$.
\begin{theorem}\label{Theorem : Kodaira Akizuki Nakano vanishing theorem}
If $E$ is positive, then
\begin{equation*}
H^{p,q}_{b}(M/\mathcal{F}, E) = 0
\end{equation*}
for $p + q > n$.
\end{theorem}
\noindent In the special case where the leaves of $\mathcal{F}$ are compact, Theorem~\ref{Theorem : Kodaira Akizuki Nakano vanishing theorem} is the Kodaira-Baily vanishing theorem~\cite{Baily}. Theorem~\ref{Theorem : Kodaira Akizuki Nakano vanishing theorem} is proved by the application of the work of El~Kacimi~Alaoui and Hector~\cite{El Kacimi Alaoui Hector} and El~Kacimi~Alaoui~\cite{El Kacimi Alaoui} on transversely elliptic operators on Riemannian foliations. Their theory allows us to reduce the proof to consider differential operators on transversals of foliations, where the classical argument of Nakano-Akizuki~\cite{Akizuki Nakano} for K\"{a}hler manifolds can be applied following Demailly~\cite{Demailly}.

\subsection{Stability of positive Sasakian metrics} 

We obtain the following corollary of Theorems~\ref{Theorem : Stability} and~\ref{Theorem : Vanishing}: Let $V$ be an open neighborhood of~$0$ in $\mathbb{R}^{\ell}$ and $\{(\mathcal{F}^{t},J^{t})\}_{t \in V}$ be a smooth family of transversely holomorphic Riemannian flows on a closed manifold~$M$. Assume that $(\mathcal{F}^{0},J^{0},g_{\nu}^{0})$ is the underlying transversely holomorphic structure of the Reeb flow of a positive Sasakian metric $({g},\eta)$.
\begin{corollary}\label{Corollary : h02 = 0}
There exists an open neighborhood $V_{1}$ of~$0$ in $V$ and a smooth family of Sasakian metrics $\{({g}^{t},\eta^{t})\}_{t \in V_{1}}$ on~$M$ such that ${g}^{t}$ is a Sasakian metric such that the underlying transversely holomorphic flow of the Reeb flow of $({g}^{t},\eta^{t})$ is $(\mathcal{F}^{t},J^{t})$ for every $t$ in $V_{1}$ and $({g}^{0},\eta^{0})=({g},\eta)$.
\end{corollary}
\noindent Corollary~\ref{Corollary : h02 = 0} is deduced from Theorems~\ref{Theorem : Stability} and~\ref{Theorem : Vanishing} as follows: Since the positivity of the anticanonical line bundles of transversely holomorphic foliations is preserved under small deformation, the assumption of the positivity of $({g},\eta)$ implies that $H_{b}^{0,2}(M/\mathcal{F}^{t})=0$ for $t$ in an open neighborhood $V'$ of~$0$ in $V$ by Theorem~\ref{Theorem : Vanishing}. Hence the basic Euler class of $(\mathcal{F}^{t},J^{t})$ is of degree~$(1,1)$ for $t$ in $V'$ by definition. Thus Theorem~\ref{Theorem : Stability} concludes the proof.

\subsection{Stability of $K$-contact structures in families of Riemannian flows}
We obtain a $K$-contact variant of Theorem~\ref{Theorem : Stability}: Let~$M$ be a closed manifold. Let $(g,\eta)$ be a $K$-contact structure on~$M$ (see Definition~\ref{Definition : K cont} for $K$-contact structures). Let $V$ be an open neighborhood of~$0$ in $\mathbb{R}^{\ell}$. Let $\{\mathcal{F}^{t}\}_{t \in V}$ be a smooth family of Riemannian flows on~$M$ such that $\mathcal{F}^{0}$ is the Reeb flow of $\eta$.
\begin{theorem}\label{Theorem : K-contact} There exists an open neighborhood $V_{1}$ of~$0$ in $V$ and a smooth family $\{(g^{t},\eta^{t})\}_{t \in V_{1}}$ of $K$-contact structures on~$M$ such that $\mathcal{F}^{t}$ is the Reeb flow of $\eta^{t}$ for every $t$ in $V_{1}$ and $\eta^{0} = \eta$.
\end{theorem}

\subsection{Moduli space of Sasakian metrics with a fixed transversely K\"{a}hler flow}\label{Section : Moduli 1}
We describes the difference of moduli spaces of Sasakian metrics and their underlying transversely K\"{a}hler flows. Let~$M$ be a closed manifold. Let $\mathcal{S}$ be the set of Sasakian metrics on~$M$. We assume that $\mathcal{S}$ is nonempty. Let $\mathcal{K}$ be the isomorphism classes of transversely K\"{a}hler flows on~$M$. There exists a natural map $\mathcal{M} \colon \mathcal{S} \to \mathcal{K}$, which maps each Sasakian metric to the isomorphism class of the underlying transversely K\"{a}hler flow. We take a point $\mathfrak{k}_{0}$ on $\mathcal{K}$. Let $\Diff(\mathfrak{k}_{0})$ be the subgroup of $\Diff(M)$ consisting of diffeomorphisms which preserve the transversely K\"{a}hler flow $\mathfrak{k}_{0}$. Let  $\Diff_{0}(\mathfrak{k}_{0})$ be the identity component of $\Diff(\mathfrak{k}_{0})$ with the Fr\'{e}chet topology. Note that $\eta-f^{*}\eta$ is a closed form for $f$ in $\Diff_{0}(\mathfrak{k}_{0})$ because $f$ preserves the transverse K\"{a}hler form $d\eta$. We define
\begin{equation}\label{Equation : Ham}
\Ham(\mathfrak{k}_{0}) = \{ f \in \Diff_{0}(\mathfrak{k}_{0}) \, | \, [\eta-f^{*}\eta] = 0 \in H^{1}(M;\mathbb{R}) \}.
\end{equation}
\noindent Then we have
\begin{theorem}\label{Theorem : H1}
There exists a homeomorphism
\begin{equation*}
\mathcal{M}^{-1}(\mathfrak{k}_{0})/\Ham(\mathfrak{k}_{0}) \longrightarrow H^{1}(M;\mathbb{R}).
\end{equation*}
\end{theorem}
\noindent Thus we have the following
\begin{corollary}\label{Corollary : H1 = 0}
Let~$M$ be a closed manifold whose first Betti number is zero. If the underlying transversely K\"{a}hler flows of two Sasakian metrics $(g_{1},\alpha_{1})$ and $(g_{2},\alpha_{2})$ are isomorphic, then $(M,g_{1},\alpha_{1})$ and $(M,g_{2},\alpha_{2})$ are isomorphic.
\end{corollary}
\noindent Note that $\Diff_{0}(\mathfrak{k}_{0})/\Ham(\mathfrak{k}_{0})$ is an abelian group (see the second paragraph of Section~\ref{Section : Moduli}). Thus $\mathcal{M}^{-1}(\mathfrak{k}_{0})/\Diff_{0}(\mathfrak{k}_{0})$ is a quotient of a vector space by an abelian group in general.

\subsection*{Acknowledgement} 
The author is grateful to Jes\'{u}s~Antonio~\'{A}lvarez~L\'{o}pez for pointing out a mistake in the previous version of this manuscript. It was in the discussion with him during the author stayed in the University of Santiago de Compostela in the spring of 2009. The author would like to express his deep gratitude to Jes\'{u}s~Antonio~\'{A}lvarez~L\'{o}pez also for his invitation, great hospitality and valuable discussion. The author would like to the first referee for a comment on the simplification of the proof of Lemma~\ref{Lemma : duality}.

\section{Sasakian structures and its underlying structures}

\subsection{Sasakian metrics}

Let~$M$ be an odd dimensional smooth manifold. We recall
\begin{definition}\label{Definition : Sasakian metrics}
A pair of a contact form $\eta$ and a Riemannian metric ${g}$ on~$M$ is a {\it Sasakian metric} on~$M$ if the Riemannian metric $r^{2} {g} + dr \otimes dr$ on $M \times \mathbb{R}_{> 0}$ is a K\"{a}hler metric with K\"{a}hler form $d(r^{2} \eta)$, where $r$ is the standard coordinate on $\mathbb{R}_{> 0}$.
\end{definition}
\noindent Basic examples of Sasakian manifolds are circle bundles associated to positive holomorphic line bundles over K\"{a}hler manifolds, links of isolated singularities of complex hypersurfaces defined by weighted homogeneous polynomials and contact toric manifolds of Reeb type (see Boyer and Galicki~\cite{Boyer Galicki 2} and Blair~\cite{Blair}). 

\subsection{Foliations with transverse structures}\label{Section : transverse structures}

The {\it Reeb flow} $\mathcal{F}$ of a Sasakian manifold $(M,\eta,g)$ is the orbit foliation of the flow on~$M$ generated by the Reeb vector field $\xi$ of $\eta$, which is defined by equations $\eta(\xi) = 1$ and $\iota_{\xi}d\eta=0$. In this Section~\ref{Section : transverse structures}, we recall the definition of certain transverse structures of $\mathcal{F}$ which are given by the Sasakian metric.

We denote the tangent bundle of $\mathcal{F}$ by $T\mathcal{F}$. By the integrability of $\mathcal{F}$, the Lie bracket on $C^{\infty}(TM)$ induces the Lie derivative with respect to vector fields tangent to the leaves
\begin{multline*}
C^{\infty} (T\mathcal{F}) \otimes C^{\infty} \Big( (TM/T\mathcal{F})^{\otimes r} \otimes (T^{*}M/T^{*}\mathcal{F})^{\otimes s} \Big) \\ \longrightarrow C^{\infty} \Big( (TM/T\mathcal{F})^{\otimes r} \otimes (T^{*}M/T^{*}\mathcal{F})^{\otimes s} \Big)
\end{multline*}
for every $r \geq 0$ and $s \geq 0$. 
\begin{definition}
A foliation $\mathcal{F}$ of~$M$ with a tensor $g_{\nu}$ in $C^{\infty}((T^{*}M/T^{*}\mathcal{F})^{\otimes 2})$ is called a {\it Riemannian foliation} of~$M$ if
\begin{enumerate}
\item $g_{\nu}$ is symmetric, positive definite and
\item $\mathcal{L}_{Z}g_{\nu}=0$ for every $Z$ in $C^{\infty}(T\mathcal{F})$.
\end{enumerate}
We call such $g_{\nu}$ a {\it transverse metric} of $\mathcal{F}$.
\end{definition}
\begin{example}
A foliation defined by a one parameter family of isometries of a Riemannian manifold is a Riemannian foliation of dimension~$1$.
\end{example}

We recall certain terminology for the definition of transversely holomorphic foliations.
\begin{definition}
A local vector field $X$ of $TM$ is {\it basic} with respect to $\mathcal{F}$ if $[Y,X]$ is tangent to $\mathcal{F}$ for every local vector field $Y$ tangent to $\mathcal{F}$. 
\end{definition}
\noindent It is easy to see that a local vector field $X$ is basic if and only if the flow generated by $X$ maps each leaf of $\mathcal{F}$ to a leaf of $\mathcal{F}$. Let $\pi : TM \to TM/T\mathcal{F}$ be the canonical projection. 
\begin{definition}
A local section $X$ of $TM/T\mathcal{F}$ is {\it transverse} with respect to $\mathcal{F}$ if $\pi(\widetilde{X}) = X$ for some basic local vector field with respect to $\mathcal{F}$.
\end{definition}
\noindent The subspace $\mathcal{X}(\mathcal{F})$ of $C^{\infty}(TM)$ consisting of basic vector fields is a Lie subalgebra of $C^{\infty}(TM)$. The subspace of $C^{\infty}(TM/T\mathcal{F})$ consisting of transverse vector fields admits a Lie bracket $[\cdot,\cdot]$ which satisfies $[\pi(\widetilde{X}),\pi(\widetilde{Y})] = \pi[\widetilde{X},\widetilde{Y}]$ for any basic vector fields $\widetilde{X}$ and $\widetilde{Y}$, because $C^{\infty}(T\mathcal{F})$ is an ideal of $\mathcal{X}(\mathcal{F})$. We recall
\begin{definition}
A foliation $\mathcal{F}$ of~$M$ with a tensor $J$ in $\Aut(TM/T\mathcal{F})$ is called a {\it transversely holomorphic foliation} of~$M$ if
\begin{enumerate}
\item $J^{2} = - \id$, 
\item $\mathcal{L}_{Z}J=0$ for every $Z$ in $C^{\infty}(T\mathcal{F})$ and
\item the Nijenhaus tensor $N_{J}(X,Y)=[JX,JY]-J[JX,Y]-J[X,JY]-[X,Y]$ of $J$ vanishes on any local transverse vector fields $X$ and $Y$ with respect to $\mathcal{F}$.
\end{enumerate}
We call such $J$ a {\it transverse complex structure} of $\mathcal{F}$.
\end{definition}
\begin{example}\label{Example : Fstd}
Let $B^{n}$ be the unit disk of $\mathbb{C}^{n}$. Consider a foliation $\mathcal{F}_{\std}$ on $\mathbb{R}^{m} \times B^{n}$ given by a decomposition $\mathbb{R}^{m} \times B^{n} = \sqcup_{z \in B^{n}} \mathbb{R}^{m} \times \{z\}$. This foliation $\mathcal{F}_{\std}$ has a transverse complex structure $J_{\std}$ is given by the complex structure of $\mathbb{C}^{n}$. We call $(\mathbb{R}^{m} \times B^{n},\mathcal{F}_{\std},J_{\std})$ the {\it standard transversely holomorphic foliation} of $\mathbb{R}^{m} \times B^{n}$.
\end{example}
\begin{example}
The total space of a fiber bundle over a complex manifold has a transversely holomorphic Riemannian foliation whose leaves are fibers.
\end{example}
\noindent A fundamental result of G\'{o}mez Mont~\cite{Gomez Mont} is as follows:
\begin{theorem}\label{Theorem : GM}
Any complex codimension $n$ transversely holomorphic foliation of dimension $m$ is locally isomorphic to $(\mathbb{R}^{m} \times B^{n},\mathcal{F}_{\std},J_{\std})$.
\end{theorem}
\noindent By this result, transversely holomorphic foliations are described in terms of Haefliger~$1$-cocycles (see Section~\ref{Section : Atlas} for such description).

The following is relevant in Sasakian geometry:
\begin{definition}\label{Definition : tr K}
A triple $(\mathcal{F},J,g_{\nu})$ is called a {\it transversely K\"{a}hler foliation} of~$M$ if
\begin{enumerate}
\item $(\mathcal{F},g_{\nu})$ is a Riemannian foliation of~$M$,
\item $(\mathcal{F},J)$ is a transversely holomorphic foliation of~$M$ and
\item the tensor field $\omega$ defined by $\omega(X,Y) = g_{\nu}(X,JY)$ is antisymmetric and closed when regarded as a~$2$-form on~$M$ by the injection $\wedge^{2} (T^{*}M/T^{*}\mathcal{F}) \to \wedge^{2} T^{*}M$.
\end{enumerate}
We call the~$2$-form $\omega$ in (iii) the {\it transverse K\"{a}hler form} of $(\mathcal{F},J,g_{\nu})$.
\end{definition}
\noindent We will regard the transverse K\"{a}hler form $\omega$ of a transversely K\"{a}hler flow as a~$2$-form on~$M$ by the injection $\wedge^{2} (T^{*}M/T^{*}\mathcal{F}) \to \wedge^{2} T^{*}M$ throughout this article. 

Boyer, Galicki and Nakamaye~\cite{Boyer Galicki Nakamaye 2} proved that the Reeb flow of a Sasakian manifold has a transversely K\"{a}hler structure:
\begin{lemma}\label{Lemma : Sasakian to foliations}
The Reeb flow of a Sasakian manifold $(M,\eta,g)$ is transversely K\"{a}hler whose transverse K\"{a}hler form is $d\eta$.
\end{lemma}

\subsection{$K$-contact structures}

Let $\eta$ be a contact form on a closed manifold~$M$ with a Riemannian metric $g$. We recall 
\begin{definition}\label{Definition : K cont}
$(M,\eta,g)$ is called a {\em $K$-contact manifold} if
\begin{enumerate}
\item $g$ is preserved by the Reeb flow of $\eta$ and
\item there exists an almost complex structure $\Phi $ on $\ker \eta$ such that $g(X,Y) = d\eta(X,\Phi Y)$ for $X$, $Y$ in $C^{\infty}(\ker \eta)$.
\end{enumerate}
\end{definition}
\noindent A Sasakian structure is a $K$-contact structure (see Proposition~6.5.12 of \cite{Boyer Galicki}). In particular, the Reeb flow of $\eta$ is Riemannian by the condition (i). The converse is true by an observation of Yamazaki:
\begin{proposition}[Proposition~2.1 of Yamazaki~\cite{Yamazaki}]\label{prop : Yam}
A closed manifold with a contact form whose Reeb flow is a Riemannian flow admits a structure of $K$-contact manifold.
\end{proposition}
The characterization of Sasakian metrics in terms of $K$-contact structure in the last sentence of Section~6.4 of Blair~\cite{Blair} (see also Proposition~6.5.14 of Boyer and Galicki~\cite{Boyer Galicki}) is as follows:
\begin{proposition}\label{prop : CRKS}
A $K$-contact structure is a Sasakian metric if the $CR$-structure $(\ker \eta, \Phi)$ is integrable.
\end{proposition}
By Propositions~\ref{prop : Yam} and~\ref{prop : CRKS}, we get
\begin{lemma}\label{Lemma : CR K}
A pair of an integrable $CR$-structure $(H,\Phi )$ and a contact form $\eta$ determine a Sasakian metric if $H=\ker \eta$ and the Reeb flow of $\eta$ is Riemannian.
\end{lemma}

\subsection{A characterization of Sasakian structures}

We have the following characterization of Sasakian metrics in terms of transversely K\"{a}hler flows and contact forms:
\begin{proposition}\label{Proposition : Transversely Kahler}
A pair of a transversely K\"{a}hler flow $(\mathcal{F},J,g_{\nu})$ and a contact form $\eta$ determines a Sasakian structure on~$M$ if $d\eta = \omega$, where $\omega$ is the transverse K\"{a}hler form of $(\mathcal{F},J,g_{\nu})$.
\end{proposition}

\begin{remark}
Conversely, the contact form and the underlying transversely K\"{a}hler structure of the Reeb flow of a Sasakian manifold satisfy the relation in Proposition~\ref{Proposition : Transversely Kahler}.
\end{remark}

\begin{proof}[Proof of Proposition~\ref{Proposition : Transversely Kahler}]
By Lemma~\ref{Lemma : CR K}, it is suffices to show that a pair of a transversely K\"{a}hler flow $(\mathcal{F},J,g_{\nu})$ and a contact form $\eta$ such that $d\eta = \omega$ determines a pair of a $CR$-structure and a contact form which satisfy the conditions of Lemma~\ref{Lemma : CR K}. Let $H=\ker \eta$ and 
\begin{equation*}
\pi \colon TM \to TM/T\mathcal{F}
\end{equation*}
the canonical projection. Letting $\Phi =(\pi|_{H})^{-1} \circ J \circ (\pi|_{H})$, we have a $CR$-structure $(H,\Phi )$ on~$M$. Clearly the Reeb flow of $\eta$ is Riemannian, because the Reeb flow of $\eta$ is transversely K\"{a}hler by the assumption. Thus the unique nontrivial part of the proof is the integrability of the $CR$-structure $(H,\Phi )$.

By $H=\ker \eta$, we get
\begin{equation*}
\frac{1}{2} \eta([X,Y] - [\Phi X,\Phi Y]) = d\eta (X,Y) - d\eta (\Phi X,\Phi Y)
\end{equation*}
for local sections $X$ and $Y$ of $H$. The right hand side is zero, because $d\eta$ is a transverse K\"{a}hler form, which is $\Phi $-invariant. It follows that $[\Phi X,\Phi Y] - [X,Y]$ is a local section of $H$ for any local sections $X$ and $Y$. This implies the first half of the integrability condition. 

We will show that the Nijenhaus tensor
\begin{equation*}
N_{\Phi}(X,Y) = [\Phi X,\Phi Y] - \Phi ( [\Phi X,Y] + [X,\Phi Y] ) - [X,Y] 
\end{equation*}
vanishes for any local sections $X$ and $Y$ of $H$. Fix a point $x$ on~$M$. Let $(\mathbb{R} \times B^{n},\mathcal{F}_{\std},J_{\std})$ be the standard transversely holomorphic foliation of $\mathbb{R} \times B^{n}$ of dimension~$1$ mentioned in Example~\ref{Example : Fstd}). By a theorem of G\'{o}mez Mont (see Theorem~\ref{Theorem : GM}), we have an open neighborhood $U_{x}$ of $x$ such that $(U_{x},\mathcal{F}|_{U_{x}},J|_{U_{x}})$ is isomorphic to $(\mathbb{R} \times B^{n},\mathcal{F}_{\std},J_{\std})$. We identify $U_{x}$ with $\mathbb{R} \times B^{n}$ by this isomorphism. By the product structure of $\mathbb{R} \times B^{n}$, we can decompose a vector field $X$ on $\mathbb{R} \times B^{n}$ as $X = X_{1} + X_{2}$, where $X_{1}$ is the $T\mathbb{R}$-component and $X_{2}$ is the $TB^{n}$-component. We take a section $X$ of $H|_{U_{x}}$ so that the $TB^{n}$-component $X_{2}$ of $X$ is projectable to a linear vector field on $B^{n}$. We take another section $Y$ of $H|_{U_{x}}$ in the same way. Clearly $X$ and $Y$ are basic vector fields. Let $H'$ be the kernel of the differential map of the first projection $\mathbb{R} \times B^{n} \to \mathbb{R}$. Let $\Phi' =(\pi|_{H'})^{-1} \circ J \circ (\pi|_{H'})$. Since the $TB^{n}$-components of $X$ and $Y$ are linear, we get $[X, \Phi' Y] \in C^{\infty}(T\mathcal{F})$. Since $(\Phi - \Phi') Y \in C^{\infty}(T\mathcal{F})$ and $X$ is basic, we get $[X, (\Phi - \Phi')Y] \in C^{\infty}(T\mathcal{F})$. By $[X, \Phi Y] = [X, (\Phi - \Phi')Y] + [X, \Phi'Y]$, it follows that $[X, \Phi Y] \in C^{\infty}(T\mathcal{F})$. Similarly, $[\Phi X,Y] \in C^{\infty}(T\mathcal{F})$. Thus we have
\begin{equation*}
\Phi ([\Phi X,Y]+[X,\Phi Y]) = 0.
\end{equation*}
Hence we have
\begin{equation*}
N_{\Phi}(X,Y) = [\Phi X,\Phi Y] - [X,Y].
\end{equation*}
Thus, by the last paragraph, $N_{\Phi}(X,Y)$ is a section of $H|_{U_{x}}$. On the other hand, by the integrability of $J$, we have
\begin{align*}
\pi(N_{\Phi}(X,Y)) = N_{J}(\pi(X),\pi(Y)) = 0.
\end{align*}
Thus $N_{\Phi}(X,Y)$ is tangent to $\mathcal{F}$. Since $N_{\Phi}(X,Y)$ is a section of $H|_{U_{x}}$ tangent to $\mathcal{F}$, we have $N_{\Phi}(X,Y)=0$.

The value of $N_{\Phi}(X,Y)$ at $x$ is determined only by vectors $X_{x}$ and $Y_{x}$. Thus the argument above shows $N_{\Phi}(X,Y)=0$ for any local section $X$ and $Y$ of $H$. Hence $(H,\Phi )$ is integrable.
\end{proof}

\subsection{Families of flows with transverse structures}

Let $V$ be an open subset of $\mathbb{R}^{\ell}$. We define smooth families of flows with various transverse structures whose parameter space is $V$ to give the formalism to consider the deformation of flows. Note that our definition of smooth families of flows with transverse structures is equivalent to that of smooth families of deformation given in terms of families of~$1$-cocycles used in Kodaira-Spencer theory~\cite{Kodaira Spencer 2} or deformation theory of transversely holomorphic foliations~\cite{Girbau Haefliger Sundararaman}.

\begin{definition}
\begin{enumerate}
\item A {\it smooth family of flows} on~$M$ over $V$ is a flow $\mathcal{F}^{\amb}$ on a smooth manifold $M \times V$ such that $M \times \{t\}$ is saturated by the leaves of $\mathcal{F}^{\amb}$. We denote the restriction of $\mathcal{F}^{\amb}$ to $M \times \{t\}$ by $\mathcal{F}^{t}$. We will denote such family of flows on~$M$ by $\{\mathcal{F}^{t}\}_{t \in V}$.
\item For a smooth family $\{\mathcal{F}^{t}\}_{t \in V}$ of flows on~$M$, the kernel of the differential map of the second projection $\pr_{2} \colon T(M \times V)/T\mathcal{F}^{\amb} \to TV$ is called the {\it family of normal bundles} of $\{\mathcal{F}^{t}\}_{t \in V}$.
\item A {\it smooth family of Riemannian flows} on~$M$ is a pair of a smooth family of flows $\mathcal{F}^{t}$ and a smooth Riemannian metric $g_{\nu}^{\amb}$ on the family of normal bundles of $\{\mathcal{F}^{t}\}_{t \in V}$ such that $(\mathcal{F}^{t},g^{\amb}|_{M \times \{t\}})$ is a Riemannian flow on~$M$ for each $t$ in $V$.
\item A {\it smooth family of transversely holomorphic flows} on~$M$ is a pair of a smooth family of flows $\{\mathcal{F}^{t}\}_{t \in V}$ and a complex structure $J^{\amb}$ on the family of normal bundles of $\{\mathcal{F}^{t}\}_{t \in V}$ such that $(\mathcal{F}^{t},J^{\amb}|_{M \times \{t\}})$ is a transversely holomorphic flow on~$M$ for each $t$ in $V$.
\end{enumerate}
\end{definition}
\noindent A smooth family of transversely K\"{a}hler flows on~$M$ is similarly defined.

\section{Basic Euler classes of transversely holomorphic isometric flows}

\subsection{Isometric flows}

The Reeb flow of a Sasakian manifold preserves the Riemannian metric as it is well known (see Proposition~6.5.6 of Boyer and Galicki \cite{Boyer Galicki}). We recall
\begin{definition}\label{Definition : Isom}
A Riemannian flow $(\mathcal{F},g_{\nu})$ on~$M$ is called {\it isometric} if there exist a Riemannian metric ${g}$ on~$M$ and a nowhere vanishing vector field $\xi$ tangent to $\mathcal{F}$ such that the flow generated by $\xi$ preserves ${g}$. 
\end{definition}
\noindent We will call $({g},\xi)$ in this definition a {\it Killing pair} on $\mathcal{F}$. A Riemannian metric ${g}$ on~$M$ is called a {\it Killing metric} on $(M,\mathcal{F})$, if there exists a vector field $\xi$ such that $({g},\xi)$ is a Killing pair on $\mathcal{F}$. Note that the terminology ``Killing pair'' and ``Killing metric'' are not standard. We recall also
\begin{definition}\label{Definition : GT}
A foliated manifold $(M,\mathcal{F})$ is {\it geometrically taut} if there exists a Riemannian metric ${g}$ on~$M$ such that every leaf of $\mathcal{F}$ is a minimal submanifold of $(M,g)$. Such $g$ is called a {\it minimal metric} on $(M,\mathcal{F})$.
\end{definition}
\noindent Geometrically tautness of Riemannian foliations was characterized in terms of basic cohomology by a theorem of Masa~\cite{Masa} in general, or a theorem of Molino and Sergiescu~\cite{Molino Sergiescu} in the case of flows. By the following lemma due to Carri\`{e}re~\cite{Carriere}, the isometricity of oriented Riemannian flows is equivalent to the geometrically tautness:
\begin{lemma}\label{Lemma : minimal and Killing}
Let $(M,\mathcal{F})$ be a closed manifold with an oriented Riemannian flow. Then a minimal metric $g$ is Killing if and only if ${g}$ is bundle-like.
\end{lemma}
\noindent Recall that a Riemannian metric $g$ on $(M,\mathcal{F})$ is called {\it bundle-like} if the metric induced on $TM/T\mathcal{F}$ from $g$ through the identification $TM/T\mathcal{F} \cong (T\mathcal{F})^{\perp}$ is a transverse metric of $\mathcal{F}$.

\subsection{$\mathcal{F}$-fibered Hermitian vector bundles and basic Dolbeault cohomology}\label{Section : Atlas}

El~Kacimi~Alaoui \cite{El Kacimi Alaoui} defined an $\mathcal{F}$-fibered Hermitian vector bundle on $(M,\mathcal{F})$ by a Hermitian vector bundle $(E,h)$ on $M$ such that the Chern connection is basic and $\nabla_{X}^{\otimes} h=0$ for any $X$ in $C^{\infty}(T\mathcal{F})$, where $\nabla^{\otimes}$ is the connection induced on $E^{*} \otimes \overline{E}^{*}$. We will give another equivalent definition of $\mathcal{F}$-fibered Hermitian vector bundles on foliated manifolds $(M,\mathcal{F})$ which is convenient in this article. We use $\mathbb{C}$ as the coefficient field of differential forms throughout this Section~\ref{Section : Atlas}.

Let $(\mathcal{F},J)$ be a transversely holomorphic foliation of real dimension $m$ and complex codimension $n$ on a closed manifold~$M$. Let $B^{n}$ be the unit disk in $\mathbb{C}^{n}$. Let $(\mathcal{F}_{\std},J_{\std})$ be the standard transversely holomorphic foliation of $\mathbb{R}^{m} \times B^{n}$ mentioned in Examples~\ref{Example : Fstd}. At each point $x$ on~$M$, we have an open neighborhood $U$ of $x$ such that $(U,\mathcal{F}|_{U},J|_{U})$ is isomorphic to $(\mathbb{R}^{m} \times B^{n},\mathcal{F}_{\std},J_{\std})$ by a theorem of G\'{o}mez Mont (see Theorem~\ref{Theorem : GM}). So we take a covering $\{U_{j}\}$ of~$M$ such that $(U_{j},\mathcal{F}|_{U_{j}},J|_{U_{j}})$ is isomorphic to $(\mathbb{R}^{m} \times B^{n}_{j},\mathcal{F}_{\std},J_{\std})$ as transversely holomorphic foliations, where $B^{n}_{j}$ is the unit ball in $\mathbb{C}^{n}$. We denote the composite
\begin{equation*}
U_{j} \longrightarrow \mathbb{R}^{m} \times B^{n}_{j} \longrightarrow B^{n}_{j}
\end{equation*}
by $\phi_{j}$, where the second map is the second projection.
\begin{definition}\label{Definition : F fibered vector bundle}
A Hermitian vector bundle $(E,h_{E})$ on~$M$ is called $\mathcal{F}$-{\it fibered} if there exists a Hermitian vector bundle $(E_{j},h_{j})$ on $B_{j}^{n}$ such that $(E|_{U_{j}},h_{E}|_{U_{j}}) = (\phi_{j}^{*}E_{j},\phi_{j}^{*}h_{j})$ for every $j$. An $\mathcal{F}$-fibered Hermitian vector bundle is {\it holomorphic} if such $(E_{j},h_{j})$ can be taken so that transition functions $E_{j}|_{\phi_{j}(U_{j} \cap U_{j'})} \to E_{j'}|_{\phi_{j'}(U_{j} \cap U_{j'})}$ are holomorphic on $\phi_{j}(U_{j} \cap U_{j'})$ for each $j$ and $j'$ satisfying $U_{j} \cap U_{j'} \neq \emptyset$.
\end{definition}
\noindent This definition of $\mathcal{F}$-fibered Hermitian vector bundles is equivalent to the original definition of El~Kacimi~Alaoui as follows: 
\begin{lemma}
For a Hermitian vector bundle $(E,h)$, the Chern connection is basic and $\nabla_{X}^{\otimes} h=0$ for any $X$ in $C^{\infty}(T\mathcal{F})$ if and only if $(E,h)$ is $\mathcal{F}$-fibered in the meaning of Definition \ref{Definition : F fibered vector bundle}.
\end{lemma}
\begin{proof}
Let $\pi \colon P \to M$ be the $\GL(r,\mathbb{C})$-principal bundle of $E$, where $r$ is the rank of $E$. By Example 5.2 in page 76 of Kobayashi and Nomizu \cite{Kobayashi Nomizu}, we can identify $h$ with a function $f \colon P \to \mathbb{C}^{r *} \otimes \overline{\mathbb{C}^{r}}^{*}$ such that $f(ug) = \rho(g^{-1}) f(u)$ for any $g$ in $\GL(r,\mathbb{C})$ and any $u$ in $P$, where $\rho \colon \GL(r,\mathbb{C}) \to \GL(\mathbb{C}^{r *} \otimes \overline{\mathbb{C}^{r}}^{*})$ is the canonical action. By Lemma in the page 115 of \cite{Kobayashi Nomizu}, we have $\nabla^{\otimes}_{X}h (\pi(u)) = u(\widetilde{X} f)$ for a point $u$ on $P$, and a vector field $X$, where $u$ is regarded as an isomorphism $\mathbb{C}^{r *} \otimes \overline{\mathbb{C}^{r}}^{*} \to E_{\pi(u)}$ and $\widetilde{X}$ is the horizontal lift of $X$. The formula implies the equivalence.
\end{proof}
  
\begin{example}
$\wedge^{k}(T^{*}M/T^{*}\mathcal{F})^{1,0} \to M$ is a holomorphic $\mathcal{F}$-fibered vector bundle on $(M,\mathcal{F})$.
\end{example}

A foliated manifold has a cohomology called basic cohomology, which can be regarded as the de Rham cohomology of the leaf space in a sense. We refer to Masa~\cite{Masa} or El~Kacimi~Alaoui and Hector~\cite{El Kacimi Alaoui Hector} for the basic cohomology of Riemannian foliations. We will recall the definition of a Dolbeault version of basic cohomology for holomorphic $\mathcal{F}$-fibered vector bundles $E$.
\begin{definition}
A differential form $\alpha$ on~$M$ with values in $E$ is a {\it basic} $(p,q)$-{\it form on} $(M,\mathcal{F})$ {\it with values in} $E$ if there exists a $(p,q)$-form $\alpha_{j}$ on $B^{n}_{j}$ with values in $E_{j}$ such that $\alpha|_{U_{j}} = \phi_{j}^{*}\alpha_{j}$ for every $j$. We denote the space of basic $(p,q)$-forms on $(M,\mathcal{F},J)$ with values in $E$ by $\Omega_{b}^{p,q}(M/\mathcal{F},E)$.
\end{definition}
\noindent This notion are also independent of the foliated atlas $\{U_{j}\}$. In the case where $E$ is the trivial bundle of rank $1$, it is easy to see that the space of basic $k$-form $\Omega^{k}_{b}(M/\mathcal{F})$ is given by
\begin{equation}
\Omega^{k}_{b}(M/\mathcal{F}) = \{ \alpha \in \Omega^{k}(M) \, | \, \iota_{X}\alpha = 0, \iota_{X}d\alpha = 0, \forall X \in C^{\infty}(T\mathcal{F}) \}.
\end{equation}
We also recall
\begin{definition}
A section $s$ of $E$ is {\it basic} if there exists a section $s_{j}$ of $E_{j}$ such that $s|_{U_{j}} = \phi_{j}^{*}s_{j}$.
\end{definition}

Let $(E,h_{E})$ be a holomorphic $\mathcal{F}$-fibered Hermitian vector bundle on $(M,\mathcal{F})$. 
\begin{definition}\label{Definition : Dol}
The {\it basic Dolbeault operator}
\begin{equation*}
\overline{\partial}_{E} \colon \Omega_{b}^{p,q}(M/\mathcal{F},E) \to \Omega_{b}^{p,q+1}(M/\mathcal{F},E)
\end{equation*}
is defined by
\begin{equation}\label{Equation : partial E}
(\overline{\partial}_{E} \alpha)|_{U_{j}} = \phi_{j}^{*} \Big( \sum_{i} \big( \overline{\partial}_{B^{n}} \beta_{i} \big) \otimes s_{i} \Big)
\end{equation}
for $\alpha$ in $\Omega_{b}^{p,q}(M/\mathcal{F},E)$, where $s_{i}$ is a local holomorphic section of $E_{j}$, and $\alpha|_{U_{j}}$ is written as $\alpha|_{U_{j}} = \phi_{j}^{*} (\sum_{i} \beta_{i} \otimes s_{i})$, and $\overline{\partial}_{B^{n}}$ is the Dolbeault operator on $B^{n}$.
\end{definition}
Then $(\Omega_{b}^{p,\bullet}(M/\mathcal{F},E), \overline{\partial}_{E})$ is a differential complex.
\begin{definition}
The $p$-{\it th basic Dolbeault cohomology} $H^{p,\bullet}_{b}(M/\mathcal{F},E)$ {\it of} $(\mathcal{F},J)$ {\it with values in} $E$ is defined by
\begin{equation*}
H^{p,\bullet}_{b}(M/\mathcal{F},E) = H^{\bullet}(\Omega_{b}^{p,\bullet}(M/\mathcal{F},E), \overline{\partial}_{E}).
\end{equation*}
\end{definition}

\begin{remark}\label{Remark : sheaf}
If the leaves of $\mathcal{F}$ are not closed, the sheaf of $(p,q)$-forms with values in $E$ may not be acyclic. Hence we may not have an isomorphism $H^{p,q}_{b}(M/\mathcal{F},E)$ between $H^{q}(M,\Omega^{p}_{b} \otimes E)$, where $\Omega^{p}_{b}$ is the sheaf of holomorphic basic $p$-forms on $(M,\mathcal{F},J)$. Here the situation is different from the case of complex manifolds or complex orbifolds, where we always have $H^{p,q}_{b}(M/\mathcal{F},E) \cong H^{q}(M,\Omega^{p}_{b} \otimes E)$.
\end{remark}

\subsection{Basic Euler classes of isometric flows}

Let $(\mathcal{F},g_{\nu})$ be an isometric Riemannian flow on a closed smooth manifold~$M$. We recall the definition of the basic Euler classes of $(M,\mathcal{F})$ due to Saralegui~\cite{Saralegui}. We denote the space of basic $k$-forms on $(M,\mathcal{F})$ of complex coefficient by $\Omega_{b}^{k}(M/\mathcal{F})_{\mathbb{C}}$. We denote the basic cohomology of $(M,\mathcal{F})$ by $H^{\bullet}_{b}(M/\mathcal{F})$.

\begin{definition}[Saralegui~\cite{Saralegui}]\label{def:bE}
Let $({g},\xi)$ be a Killing pair on $\mathcal{F}$. We define a~$1$-form $\eta$ on~$M$ by $\eta(Y)={g}(\xi,Y)$. Then $d\eta$ is a basic~$2$-form on $(M,\mathcal{F})$. The {\it basic Euler class} of $\mathcal{F}$ is defined by $\mathbb{R}^{\times}[d\eta]$ in $H^{2}_{b}(M/\mathcal{F})$ up to multiplication of nonzero real numbers.
\end{definition}
\noindent Saralegui~\cite{Saralegui} proved that the basic Euler class of $\mathcal{F}$ depends only on the smooth type of the flow $\mathcal{F}$ (see Royo Prieto~\cite{Royo Prieto} for the generalization of the definition of basic Euler classes for Riemannian flows). 

\begin{example}
If $\mathcal{F}$ is an isometric flow defined by fibers of a circle bundle, the basic cohomology of $\mathcal{F}$ coincides with the de Rham cohomology of the base manifold. In this case, the basic Euler class of $\mathcal{F}$ coincides with the Euler class of the circle bundle up to multiplication of real numbers.
\end{example}

\subsection{The~$(0,2)$-component of the basic Euler class of transversely holomorphic flows}\label{sec:HF}

Let~$M$ be a closed smooth manifold. Let $(\mathcal{F},g_{\nu},J)$ be a transversely holomorphic Riemannian flow on~$M$. We assume that $(\mathcal{F},g_{\nu})$ is isometric with a Killing pair $({g},\xi)$. Let $\eta$ be the characteristic form of $(M,\mathcal{F},g)$, which is defined by $\eta(X) = g(\xi,X)$. Then the basic cohomology class of $d\eta$ is the basic Euler class of $(M,\mathcal{F})$ by Definition \ref{def:bE}. Note that the~$(0,2)$-component $(d\eta)^{0,2}$ of $d\eta$ is $\overline{\partial}$-closed, because $\overline{\partial}(d\eta)^{0,2} = (dd\eta)^{0,3}=0$.
\begin{definition} We define the~$(0,2)$-{\it component of the basic Euler class} of $(\mathcal{F},g,J)$ by $\mathbb{R}^{\times}[(d\eta)^{0,2}]$ in $H^{0,2}_{b}(M/\mathcal{F})$ up to multiplication of nonzero real numbers. If the~$(0,2)$-component of the basic Euler class of $(\mathcal{F},g,J)$ is trivial, we say the basic Euler class of $(\mathcal{F},g,J)$ is {\it of degree}~$(1,1)$.
\end{definition}
\noindent We will show the well definedness of the~$(0,2)$-component of the basic Euler class in the following. We will use the leafwise cohomology of $(M,\mathcal{F})$. Let $\Omega^{k}(\mathcal{F})_{\mathbb{C}} = C^{\infty}(\wedge^{k}T^{*}\mathcal{F}) \otimes \mathbb{C}$. The differential on each leaf of $\mathcal{F}$ yields the leafwise differential
\begin{equation*}
d_{\mathcal{F}} : \Omega^{k}(\mathcal{F})_{\mathbb{C}} \to \Omega^{k+1}(\mathcal{F})_{\mathbb{C}}.
\end{equation*}
The cohomology of the differential complex $( \Omega^{\bullet}(\mathcal{F}),d_{\mathcal{F}})$ is called the {\it leafwise cohomology} of $(M,\mathcal{F})$. We denote it by $H^{\bullet}(\mathcal{F})$. 

\begin{remark}
The leafwise cohomology may be of infinite dimension even for linear flows on tori (see El~Kacimi~Alaoui~\cite{El Kacimi Alaoui}).
\end{remark}

The restriction $\eta|_{T\mathcal{F}}$ of a characteristic form $\eta$ of $\mathcal{F}$ to $T\mathcal{F}$ satisfies $d_{\mathcal{F}}(\eta|_{T\mathcal{F}})=0$, thus it determines a leafwise cohomology class $[\eta|_{T\mathcal{F}}]$ in $H^{1}(\mathcal{F})$. We get
\begin{lemma}\label{Lemma : KT}
Let $\eta_{j}$ be the characteristic forms of $(M,\mathcal{F},{g}_{j})$ for a Killing metric ${g}_{j}$ for $j=1$ and~$2$. Then there exists a nonzero real number $r$ such that 
\begin{equation*}
[\eta_{1}|_{T\mathcal{F}}] = r[\eta_{2}|_{T\mathcal{F}}]
\end{equation*}
in $H^{1}(\mathcal{F})$.
\end{lemma}

\begin{proof}
$H^{1}(\mathcal{F})$ can be identified with the $(0,1)$-th $E_{1}$-term $E^{0,1}_{1}(\mathcal{F})$ of the spectral sequence canonically associated to $\mathcal{F}$ (see Kamber and Tondeur~\cite{Kamber Tondeur}). We can identify the $(0,1)$-th $E_{2}$-term $E^{0,1}_{2}(\mathcal{F})$ with a subspace of $E^{0,1}_{1}(\mathcal{F})$, which consists of $d_{1}$-closed elements. By Corollary~4.7 of Kamber and Tondeur~\cite{Kamber Tondeur}, we get $\dim E^{0,1}_{2}(\mathcal{F})=1$ and $E^{0,1}_{2}(\mathcal{F})$ is generated by the leafwise cohomology class of the characteristic form of any minimal metric under these identification. Since $g_{j}$ is minimal by Lemma~\ref{Lemma : minimal and Killing} for $j=1$ and~$2$, we get nonzero real number $r$ such that $[\eta_{1}|_{T\mathcal{F}}] = r[\eta_{2}|_{T\mathcal{F}}]$ in $H^{1}(\mathcal{F})$.
\end{proof}

Let $\Diff(M,\mathcal{F})$ be the subgroup of $\Diff(M)$ consisting of diffeomorphisms which map each leaf of $\mathcal{F}$ to itself. Let $\Diff_{0}(M,\mathcal{F})$ be the identity component of  $\Diff(M,\mathcal{F})$ with respect to the Fr\'{e}chet topology.
\begin{lemma}\label{Lemma : Moser}
Let $\eta_{j}$ be the characteristic form of $(M,\mathcal{F},{g}_{j})$ for a Killing metric ${g}_{j}$ for $j=1$ and~$2$. There exists $r$ in $\mathbb{R}$ and $f$ in $\Diff_{0}(M,\mathcal{F})$ such that
\begin{equation*}
\eta_{1}|_{T\mathcal{F}} = r f^{*}( \eta_{2}|_{T\mathcal{F}}).
\end{equation*}
\end{lemma}
\begin{proof}
By Lemma~\ref{Lemma : KT}, we have a real number $r$ such that $[\eta_{1}|_{T\mathcal{F}}] = r[\eta_{2}|_{T\mathcal{F}}]$ in $H^{1}(\mathcal{F})$. By the leafwise version of the Moser's argument (see Hector, Macias and Saralegui~\cite{Hector Macias Saralegui}), we have a diffeomorphism $f$ of~$M$ in $\Diff_{0}(M,\mathcal{F})$ such that $\eta_{1}|_{T\mathcal{F}} = r f^{*} (\eta_{2}|_{T\mathcal{F}})$.
\end{proof}

\begin{lemma}
The~$(0,2)$-component of the basic Euler class is well-defined for $(\mathcal{F},J)$ up to multiplication of nonzero real numbers.
\end{lemma}
\begin{proof}
Let $({g}_{1},\xi_{1})$ and $({g}_{2},\xi_{2})$ be two Killing pairs on $(M,\mathcal{F})$. Let $\eta_{j}$ be the characteristic form of $(M,\mathcal{F},{g}_{j})$ for $j=1$ and~$2$. By Lemma~\ref{Lemma : Moser}, there exist a real number $r$ and a diffeomorphism $f$ in $\Diff_{0}(M,\mathcal{F})$ such that $\eta_{1}|_{T\mathcal{F}} = r f^{*} (\eta_{2}|_{T\mathcal{F}})$. Thus $f_{*} \xi_{1} = r \xi_{2}$. Here $\eta_{1} - r f^{*} \eta_{2}$ is a basic~$1$-form by $(\eta_{1} - r f^{*} \eta_{2})|_{T\mathcal{F}}=0$ and $\mathcal{L}_{\xi_{1}}(\eta_{1} - r f^{*} \eta_{2})=0$. Hence we get $[(d\eta_{1})^{0,2}] = r f^{*}[(d\eta_{2})^{0,2}]$ by
\begin{equation*}
\overline{\partial} \big( (\eta_{1} - r f^{*}\eta_{2})^{0,1} \big) = (d\eta_{1})^{0,2} - r f^{*}(d\eta_{2})^{0,2}.
\end{equation*}
Since $f$ is belongs to $\Diff_{0}(M,\mathcal{F})$, Proposition~II.1.c of El~Kacimi~Alaoui~\cite{El Kacimi Alaoui 2} implies that $f$ trivially acts on the leafwise cohomology. Then $f^{*}[(d\eta_{2})^{0,2}]=[(d\eta_{2})^{0,2}]$. Thus we get $[(d\eta_{1})^{0,2}] = r [(d\eta_{2})^{0,2}]$, which concludes the proof.
\end{proof}

By Lemma~\ref{Lemma : Sasakian to foliations}, the underlying transversely holomorphic Riemannian flow of a Sasakian manifold $(M,g,\eta)$ is transversely K\"{a}hler with transverse K\"{a}hler form $d\eta$. Since a transverse K\"{a}hler form is always of degree~$(1,1)$, we have
\begin{lemma}\label{Lemma : Sasakian case}
The basic Euler classes of the underlying transversely holomorphic isometric Riemannian flows of Sasakian manifolds are of degree~$(1,1)$.
\end{lemma}

\section{Existence of extensions of Sasakian metrics}

\subsection{A differential operator $D$}\label{sec:D}

Let $(M,\mathcal{F},J,g_{\nu})$ be a closed manifold with a transversely holomorphic Riemannian flow. We extend $g_{\nu}$ to a $J$-invariant Hermitian metric on $(TM/T\mathcal{F}) \otimes \mathbb{C}$. We introduce a triple grading of $\Omega^{\bullet}(M)_{\mathbb{C}} = \Omega^{\bullet}(M) \otimes \mathbb{C}$, which is known for researchers of CR geometry. The triple grading is the complex version of the double grading mentioned in Section~\ref{sec : kappa}, which is well known for researchers of foliation theory. Then we will introduce a differential operator $D$ to use later in the proof of Theorem~\ref{Theorem : Stability}.

Let $({g},\xi)$ be a Killing pair on $\mathcal{F}$ such that the transverse metric on $TM/T\mathcal{F}$ induced from ${g}$ is equal to $g_{\nu}$. Let $\eta$ be the characteristic form of $(M,\mathcal{F},{g})$ defined by $\eta(X) = g(\xi,X)$. Let $H = \ker \eta$ and denote the orthogonal projection $TM \to H$ by $\pi$. We have a $CR$-structure $(H,\Phi )$ defined by $\Phi =(\pi|_{H})^{-1} \circ J \circ (\pi|_{H})$. Let $H^{0,1}$ (resp. $H^{1,0}$) be the vector subbundle of $H \otimes \mathbb{C}$ whose fibers are the $(-i)$-eigenspaces (resp. $i$-eigenspaces) of $\Phi $.

The coefficient ring of differential forms is $\mathbb{C}$ in this Section~\ref{sec:D}. By the decomposition $TM \otimes \mathbb{C} = H^{1,0} \oplus H^{0,1} \oplus (T\mathcal{F} \otimes \mathbb{C})$, we define a triple grading $\Omega^{\bullet}(M)_{\mathbb{C}} = \Omega_{h,j,k}$, where
\begin{equation}\label{Equation : t1}
\Omega_{h,j,k}=C^{\infty}(\wedge^{h} (H^{1,0})^{*}\otimes \wedge^{j} (H^{0,1})^{*} \otimes \wedge^{k} (T\mathcal{F} \otimes \mathbb{C})^{*}).
\end{equation}
By the integrability of $\mathcal{F}$, we can decompose the differential $d$ as
\begin{equation}\label{Equation : t2}
d = d_{0,0,1} + d_{0,1,0} + d_{1,0,0} + d_{0,2,-1} + d_{1,1,-1} + d_{2,0,-1},
\end{equation} 
where $d_{h',j',k'}$ is the $(\Omega_{h+h',j+j',k+k'})$-component of $d$ on $\Omega_{h,j,k}$. Note that the space of basic $(p,q)$-forms is embedded to $\Omega_{p,q,0}$. We define a differential operator 
\begin{equation}\label{Equation : t3}
D = d - d_{1,1,-1}.
\end{equation}
Let $D^{*}$ be the formal adjoint of $D$ with respect to the inner product on $\Omega^{\bullet}(M)_{\mathbb{C}}$ defined by $\langle \alpha_{1},\alpha_{2} \rangle = \int_{M} \alpha_{1} \wedge \overline{*} \alpha_{2}$ for $\alpha_{1}$ and $\alpha_{2}$ in $\Omega^{k}(M)_{\mathbb{C}}$. Let 
\begin{equation}\label{Equation : t4}
\Delta_{D} = D D^{*} + D^{*} D 
\end{equation}
and 
\begin{align*}
\mathbf{H}^{k}_{D} = \ker \{ \Delta_{D} \colon \Omega^{k}(M)_{\mathbb{C}} \to \Omega^{k}(M)_{\mathbb{C}}\}, \,\,\,\,
\mathbf{H}^{k} = \ker \{ \Delta \colon \Omega^{k}(M)_{\mathbb{C}} \to \Omega^{k}(M)_{\mathbb{C}}\},
\end{align*}
where $\Delta=dd^{*} + d^{*}d$.

We collect fundamental properties of $D$ as follows:
\begin{lemma}\label{lem:D}
\begin{enumerate}
\item $\Delta_{D}$ is a self-adjoint strongly elliptic operator.
\item $D^{*} = - \overline{*} D \overline{*}$.
\item Any $d$-closed form is $D$-closed, and any $d^{*}$-closed form is $D^{*}$-closed.
\item $\mathbf{H}^{k} \subseteq \mathbf{H}^{k}_{D}$.
\item $\mathbf{H}^{k}_{D} = \ker D \cap \ker D^{*}$.
\end{enumerate}
\end{lemma}

\begin{proof}
For $\alpha$ in $\Omega_{h,j,1}$, we have
\begin{multline*}
d_{1,1,-1}\alpha (X_{1}, \ldots, X_{h+1}, Y_{1}, \ldots, Y_{j+1}) = \\
\sum_{s < t} (-1)^{s(h+1+t)} \alpha( [X_{s},Y_{t}], X_{1}, \ldots, \widehat{X_{s}}, \ldots, X_{h+1}, Y_{1}, \ldots, \widehat{Y_{t}}, \ldots, Y_{j+1})
\end{multline*}
for $X_{1}$, $\ldots$, $X_{h+1}$ in $C^{\infty}(H^{1,0})$ and $Y_{1}$, $\ldots$, $Y_{j+1}$ in $C^{\infty}(H^{0,1})$. Thus $d_{1,1,-1}$ satisfies $d_{1,1,-1}(f\alpha) = f d_{1,1,-1}\alpha$ for any smooth function $f$ on~$M$, which means that $d_{1,1,-1}$ is a differential operator of degree~$0$. This implies that the symbol of $\Delta_{D}$ is equal to the symbol of $\Delta$. Thus (i) is proved. We have (ii) by integrating both sides of
\begin{equation*}
d(\alpha_{1} \wedge \overline{*} \alpha_{2}) = D \alpha_{1} \wedge \overline{*} \alpha_{2} + (-1)^{k} \alpha_{1} \wedge D \overline{*} \alpha_{2}
\end{equation*}
on~$M$ for $\alpha_{1}$ in $\Omega_{h,j,k}$ and $\alpha_{2}$ in $\Omega_{h-1,j,k}$, $\Omega_{h,j-1,k}$ or $\Omega_{h,j,k-1}$. Note that $\overline{*}(\Omega_{j,k,h}) \subseteq \Omega_{n-j,n-k,1-h}$, where $2n$ is the real codimension of $\mathcal{F}$. By $\overline{*}(\Omega_{j,k,h}) \subseteq \Omega_{n-j,n-k,1-h}$ and the definition of $D$, we get (iv). By definition of $\Delta_{D}$, it follows that $\ker D \cap \ker D^{*} \subseteq \mathbf{H}^{k}_{D}$. Conversely $\mathbf{H}^{k}_{D} \subseteq \ker D \cap \ker D^{*}$ by the equality $\langle \Delta_{D}\alpha,\alpha \rangle = \langle D\alpha , D\alpha \rangle + \langle D^{*}\alpha, D^{*}\alpha \rangle$. Thus (v) is proved.
\end{proof}

\begin{lemma}\label{Lemma : H 1}
Assume that $(\eta,{g})$ is a Sasakian metric on~$M$. Then we have an equality
\begin{equation*}
\mathbf{H}^{1}_{D} = \mathbf{H}^{1} \oplus \mathbb{C} \eta
\end{equation*}
in $\Omega^{1}(M)_{\mathbb{C}}$.
\end{lemma}

\begin{proof}
We show $\mathbf{H}^{1} \oplus \mathbb{C} \eta \subseteq \mathbf{H}^{1}_{D}$. Since $d\eta$ is a basic~$(1,1)$-form on $(M,\mathcal{F})$, we have $d\eta = d_{1,1,-1}\eta$. Thus
\begin{equation}\label{Equation : d eta}
D\eta = 0.
\end{equation}
Since $\overline{*} \eta$ is a basic volume form on $(M,\mathcal{F})$ by Equation 2.8 of Kamber and Tondeur~\cite{Kamber Tondeur}, we have $d \overline{*} \eta = 0$. It follows that
\begin{equation}\label{Equation d * eta}
D^{*} \eta = 0.
\end{equation}
Hence $\eta$ is $\Delta_{D}$-harmonic. We get $\mathbf{H}^{1} \subseteq \mathbf{H}^{1}_{D}$ by Lemma~\ref{lem:D} (iv). Since $\eta$ is not $d$-closed, $\eta$ is not contained in $\mathbf{H}^{1}$. Thus $\mathbf{H}^{1} \oplus \mathbb{C} \eta \subseteq \mathbf{H}^{1}_{D}$.

We show $\mathbf{H}^{1}_{D} \subseteq \mathbf{H}^{1} \oplus \mathbb{C} \eta$. Let $\alpha$ be a $\Delta_{D}$-harmonic~$1$-form. By Lemma~\ref{lem:D} (v), we get $D\alpha=0$ and $D^{*}\alpha=0$. Let $\{\varphi_{u}\}_{u \in \mathbb{R}}$ be the flow generated by the Reeb vector field $\xi$ of $\alpha$. Since $\varphi_{u}$ preserves any of $J$, $H$ and ${g}$, we have 
\begin{equation*}
\Delta_{D} \varphi_{u}^{*} = \varphi_{u}^{*} \Delta_{D}.
\end{equation*}
Thus $\varphi_{u}^{*}$ preserves $\mathbf{H}_{D}^{1}$. Since $\iota_{\xi}$ is the zero map on $\Omega_{j,k,0}$, we have $\iota_{\xi}d_{1,1,-1}\alpha=0$. It follows that $\iota_{\xi}d\alpha = \iota_{\xi} D\alpha = 0$. Then
\begin{equation}\label{Equation : L xi alpha}
\mathcal{L}_{\xi}\alpha = d(\alpha(\xi)).
\end{equation}
By~\eqref{Equation : L xi alpha}, we have
\begin{equation*}
\varphi_{u}^{*}\alpha - \alpha = \int_{0}^{u} \frac{d}{dr}\Big|_{r=s} \big( \varphi_{r}^{*}\alpha \big) ds = \int_{0}^{u} \varphi_{s}^{*} (\mathcal{L}_{\xi} \alpha) ds = d \left( \int_{0}^{u} \varphi_{s}^{*} \alpha(\xi) ds \right),
\end{equation*}
which implies that $\varphi_{u}^{*}\alpha - \alpha$ is an exact~$1$-form. Since $d_{1,1,-1}\overline{*}(\varphi_{u}^{*}\alpha - \alpha)=0$ because of the degree, we have
\begin{equation*}
d^{*}(\varphi_{u}^{*}\alpha - \alpha) = D^{*}(\varphi_{u}^{*}\alpha - \alpha) = 0.
\end{equation*}
By the Hodge decomposition of $\Omega^{1}(M)_{\mathbb{C}}$, a $d$-exact and $d^{*}$-closed form is zero. Hence 
\begin{equation*}
\varphi_{u}^{*}\alpha - \alpha = 0.
\end{equation*}
Thus $\alpha$ is invariant under the flow generated by $\xi$, which implies $\alpha(\xi)$ is a leafwise constant function on $(M,\mathcal{F})$. Letting $\beta = \alpha - \alpha(\xi) \eta$ and $f=\alpha(\xi)$, we decompose $\alpha$ as
\begin{equation}\label{Equation : decomposition alpha}
\alpha = \beta + f \eta.
\end{equation}
Here $\beta$ is a basic~$1$-form on $(M,\mathcal{F})$, because $\varphi_{u}^{*}\beta - \beta = 0$ and satisfies $\beta(\xi) =0$. By differentiating both sides of~\eqref{Equation : decomposition alpha} and operating $\iota_{\xi}$, we have 
\begin{equation*}
df = 0.
\end{equation*}
Thus $f$ is a constant. Since $D=d$ on $\Omega_{j,k,0}$ by definition, we have
\begin{equation*}
0  = D\alpha = D(\beta + f\eta) = d\beta + fD\eta.
\end{equation*}
By~\eqref{Equation : d eta}, we have $d\beta=0$. We have
\begin{equation*}
0 = D^{*}\alpha = - \overline{*} D \overline{*} \beta + fD^{*}\eta = - \overline{*} d \overline{*} \beta + fD^{*}\eta,
\end{equation*}
because $D=d$ on $\oplus_{j+k=2n-1} \Omega_{j,k,1}$, where $2n$ is the real codimension of $\mathcal{F}$. By~\eqref{Equation d * eta}, we have $d^{*}\beta=0$. Hence $\beta$ is a harmonic~$1$-form. Thus we get $\mathbf{H}^{1}_{D} \subseteq \mathbf{H}^{1} \oplus \mathbb{C} \eta$.
\end{proof}

\subsection{Mean curvature forms of Riemannian flows}\label{sec : kappa}

We will make an introduction to the description of mean curvature forms of Riemannian foliations by \'{A}lvarez L\'{o}pez~\cite{Alvarez Lopez} in the special case of Riemannian flows to prove a preliminary result necessary in the proof of Theorem~\ref{Theorem : Stability} in the next section. Let $(M,\mathcal{F})$ be a flow. Let $g$ be a Riemannian metric on~$M$. We recall
\begin{definition}
The {\it mean curvature form} of $(M,\mathcal{F},g)$ is a~$1$-form $\kappa$  on~$M$ defined by
\begin{equation*}
\kappa(X) = g(\xi ,\nabla_{\xi} X)
\end{equation*}
for $X$ in $C^{\infty}(TM)$, where $\nabla$ is the Levi-Civita connection of $(M,g)$ and $\xi$ is a vector field tangent to $\mathcal{F}$ such that $g(\xi,\xi)=1$.
\end{definition}
\noindent For a point $x$ on~$M$, the~$1$-form $\kappa_{x}$ is the mean curvature form of the leaf of $\mathcal{F}$ which goes through $x$. Thus $\kappa = 0$ if and only if every leaf of $\mathcal{F}$ is a geodesic with respect to $g$ under the length parametrization. 

Let $\eta$ be the characteristic form of $(M,\mathcal{F},g)$, which is defined by $\eta(X) = g(\xi,X)$. Then the Rummler's formula (see the second formula in the proof of Proposition~1 in Rummler~\cite{Rummler} or Lemma~10.5.6 of Candel and Conlon~\cite{Candel Conlon}) gives
\begin{equation}\label{eq : Rummler}
\kappa(X) = d\eta(\xi,X)
\end{equation}
for $X$ in $C^{\infty}(TM)$. Thus $\kappa$ is determined only by $g|_{T\mathcal{F} \otimes T\mathcal{F}}$ and the orthogonal plane field $(T\mathcal{F})^{\perp}$, because $\eta$ is clearly determined only by $g|_{T\mathcal{F} \otimes T\mathcal{F}}$ and $(T\mathcal{F})^{\perp}$. 

In the case where $\mathcal{F}$ is Riemannian and $g$ is bundle-like, \'{A}lvarez L\'{o}pez~\cite{Alvarez Lopez} calculated the change of $\kappa$ when $g|_{T\mathcal{F} \otimes T\mathcal{F}}$ and  $(T\mathcal{F})^{\perp}$ are changed. To state his calculation, we introduce a well known double grading of $\Omega^{\bullet}(M)$. In this Section~\ref{sec : kappa}, we consider only real differential forms. Let $Q=(T\mathcal{F})^{\perp}$ and
\begin{equation*}
A_{j,k}=C^{\infty}(\wedge^{j} Q^{*} \otimes \wedge^{k} T^{*}\mathcal{F}).
\end{equation*}
Then $\Omega^{\bullet}(M)$ is decomposed as
\begin{equation*}
\Omega^{\bullet}(M)=\oplus_{j,k} A_{j,k}.
\end{equation*}
By the integrability of $T\mathcal{F}$, we see that the differential $d$ is decomposed as
\begin{equation}\label{eq:decompositionofd}
d=d_{0,1}+d_{1,0}+d_{2,-1},
\end{equation}
where the subscripts correspond to the double grading of $\Omega^{\bullet}(M)$. Since $g$ is bundle-like, the formal adjoint $d^{*}$ of $d$ is decomposed as 
\begin{equation}\label{eq : d*}
d^{*}=d_{0,-1}^{*}+d_{-1,0}^{*}+d_{-2,1}^{*}
\end{equation}
by the double grading of $\Omega^{\bullet}(M)$ in a similar way. 

Let $g_{1}$ be another bundle-like metric on $(M,\mathcal{F})$ with characteristic form $\eta_{1}$. Let $\kappa_{1}$ be the mean curvature form of $(M,\mathcal{F},g_{1})$. Let $Q_{1}$ be the orthogonal plane field of $T\mathcal{F}$ with respect to $g_{1}$. If $Q=Q_{1}$, then there exists a smooth function $f$ on $M$ such that $\eta = e^{f}\eta_{1}$. Then, by the Rummler's formula~\eqref{eq : Rummler}, we get
\begin{proposition}[Equation~5.3 of~\'{A}lvarez L\'{o}pez~\cite{Alvarez Lopez}]\label{prop:Alvarez Lopez1}
If $Q=Q_{1}$, then
\begin{equation}\label{eq:Alvarez Lopez1}
\kappa_{1} - \kappa = d_{1,0}f
\end{equation}
for some $f$ in $A_{0,0}$. Conversely, for any $f$ in $A_{0,0}$, there exists a bundle-like metric $g_{1}$ on $(M,\mathcal{F})$ such that $Q_{1}=Q$ and~\eqref{eq:Alvarez Lopez1} is satisfied. Moreover if $f$ is $C^{\infty}$-close to~$0$, then we can take such $g_{1}$ so that $g_{1}$ is $C^{\infty}$-close to $g$.
\end{proposition}
\noindent \'{A}lvarez L\'{o}pez calculated the effect of the change of $Q_{1}$ to $\kappa_{1}$ in terms of $d_{0,-1}^{*}$ as follows:
\begin{proposition}[Proposition~4.3 of~\'{A}lvarez L\'{o}pez~\cite{Alvarez Lopez}]\label{prop:Alvarez Lopez2}
If $g_{1}|_{T\mathcal{F} \otimes T\mathcal{F}} = g|_{T\mathcal{F} \otimes T\mathcal{F}}$, then 
\begin{equation}\label{eq:Alvarez Lopez2}
\kappa_{1} - \kappa = d_{0,-1}^{*}\gamma
\end{equation}
for some $\gamma$ in $A_{1,1}$. Conversely, for any $\gamma$ in $A_{1,1}$, there exists a bundle-like metric $g_{1}$ on $(M,\mathcal{F})$ such that $g_{1}|_{T\mathcal{F} \otimes T\mathcal{F}} = g|_{T\mathcal{F} \otimes T\mathcal{F}}$ and~\eqref{eq:Alvarez Lopez2} is satisfied. Moreover, if $\gamma$ is $C^{\infty}$-close to $0$, then we can take such $g_{1}$ so that $g_{1}$ is $C^{\infty}$-close to $g$.
\end{proposition}
\noindent The last sentence of Proposition~\ref{prop:Alvarez Lopez2} follows from the proof of Proposition~4.3 of~\'{A}lvarez L\'{o}pez~\cite{Alvarez Lopez}.

Let $(M,\mathcal{F})$ be a geometrically taut Riemannian flow. Let $g$ be a bundle-like metric on~$M$. Let $*_{\mathcal{F}} : A_{0,1} \to A_{0,0}$ and $*_{\mathcal{F}} : A_{0,0} \to A_{0,1}$ be the leafwise Hodge star operators with respect to $g|_{T\mathcal{F} \otimes T\mathcal{F}}$. We extend $*_{\mathcal{F}}$ on $\Omega^{\bullet}(M)$ by the double grading of $\Omega^{\bullet}(M)$ so that $*_{\mathcal{F}}$ trivially acts on $Q^{*}$. Thus $*_{\mathcal{F}}(A_{k,1}) = A_{k,0}$ and $*_{\mathcal{F}}(A_{k,0}) = A_{k,1}$. We show
\begin{lemma}\label{Lemma : d0-1}
The restriction of $d_{0,-1}^{*} \colon A_{1,1} \to A_{1,0}$ to $\overline{d_{0,1}(A_{1,0})}$ is a bijection $\overline{d_{0,1}(A_{1,0})} \to \overline{d_{0,-1}^{*}(A_{1,1})}$.
\end{lemma}
\begin{proof}
By Theorem~2.1 of~\cite{Alvarez Lopez}, we get
\begin{equation}\label{eq:A10}
A_{1,0} =  \Omega_{b}^{1}(M/\mathcal{F}) \oplus \overline{d_{0,-1}^{*}(A_{1,1})}.
\end{equation}
By operating $*_{\mathcal{F}}$ to both sides of~\eqref{eq:A10}, we get 
\begin{equation}\label{eq:A11}
A_{1,1} =  \big( \Omega_{b}^{1}(M/\mathcal{F}) \wedge \eta \big) \oplus \overline{d_{0,1}(A_{1,0})}.
\end{equation}
The definition of basic forms implies that the kernel of $d_{0,1} : A_{1,0} \to A_{1,1}$ is the space $\Omega_{b}^{1}(M/\mathcal{F})$ of basic $1$-forms on $(M,\mathcal{F})$. By Equation~1.13 of~\cite{Alvarez Lopez}, we get $d_{0,-1}^{*} = - *_{\mathcal{F}} d_{0,1} *_{\mathcal{F}}$ on $A_{1,1}$. Thus
\begin{equation}\label{eq:kerd0-1}
\ker d_{0,-1}^{*} = \Omega_{b}^{1}(M/\mathcal{F}) \wedge \eta.
\end{equation}
The proof is concluded by~\eqref{eq:A10},~\eqref{eq:A11} and~\eqref{eq:kerd0-1}.
\end{proof}

We will use the leafwise cohomology $H^{\bullet}(\mathcal{F})$ in the proof of the following lemma (see Section~\ref{sec:HF}). As in the proof of Lemma \ref{Lemma : KT}, we will regard the $(0,1)$-th $E_{2}$-term $E^{0,1}_{2}(\mathcal{F})$ as a subspace of $H^{1}(\mathcal{F})$.
\begin{proposition}\label{Proposition : g_1}
If $\kappa$ is $C^{\infty}$-close to~$0$, then there exists a Killing metric $g_{1}$ on $(M,\mathcal{F})$ which is $C^{\infty}$-close to $g$.
\end{proposition}

\begin{proof}
First we show that there exists a bundle-like metric $g_{2}$ on $(M,\mathcal{F})$ with characteristic form $\eta_{2}$ such that $g_{2}$ is $C^{\infty}$-close to $g$ and $\eta_{2}|_{T\mathcal{F}}$ is the restriction of the characteristic form of a minimal bundle-like metric on $(M,\mathcal{F})$. By definition, there exists a minimal metric $g_{3}$ on $(M,\mathcal{F})$. Let $\eta_{3}$ be the characteristic form of $(M,\mathcal{F},g_{3})$. We will consider the leafwise cohomology classes $[\eta|_{T\mathcal{F}}]$ and $[\eta_{3}|_{T\mathcal{F}}]$ in $H^{1}(\mathcal{F})$. Here $E^{0,1}_{2}(\mathcal{F})$ is generated by $[\eta_{3}|_{T\mathcal{F}}]$ by Corollary~4.7 of Kamber and Tondeur~\cite{Kamber Tondeur}.
 Let $\overline{0}$ be the closure of $0$ in $H^{1}(\mathcal{F})$, which may be nontrivial because $H^{1}(\mathcal{F})$ may be non-Hausdorff. Since $(M,\mathcal{F})$ is geometrically taut, Proposition~3 of \'{A}lvarez L\'{o}pez~\cite{Alvarez Lopez 2} implies that $E^{1,0}_{2}(\mathcal{F}) \to H^{1}(\mathcal{F})/\overline{0}$ is surjective. Thus there exists $\varsigma$ in $C^{\infty}(T^{*}\mathcal{F})$ such that $\varsigma$ is $C^{\infty}$-close to $0$ and $[\eta|_{T\mathcal{F}}] + [\varsigma] = [\eta_{3}|_{T\mathcal{F}}]$. Here $\eta|_{T\mathcal{F}} + \varsigma$ is nonzero at every point on $M$, because $\varsigma$ is $C^{\infty}$-close to $0$. Thus we get a smooth function $f$ on $M$ such that $f\eta|_{T\mathcal{F}} = \eta|_{T\mathcal{F}} + \varsigma$. Here $f$ is $C^{\infty}$-close to $1$, because $\varsigma$ is $C^{\infty}$-close to $0$. Since $[\eta_{3}|_{T\mathcal{F}}]$ is belongs to $E_{2}^{0,1}(\mathcal{F})$, we get $[f\eta|_{T\mathcal{F}}]$ is belongs to $E_{2}^{0,1}(\mathcal{F})$. Thus Rummler-Sullivan's criterion implies that $f\eta|_{T\mathcal{F}}$ is the restriction of the characteristic form of a minimal bundle-like metric (see Sullivan~\cite{Sullivan}, Rummler~\cite{Rummler} and Theorem in the introduction of Haefliger~\cite{Haefliger 2}). We take a bundle-like metric $g_{2}$ on $(M,\mathcal{F})$ so that the characteristic form is $f\eta$ and $g_{2}|_{\ker \eta \otimes \ker \eta} = g|_{\ker \eta \otimes \ker \eta}$. Then $g_{2}$ is $C^{\infty}$-close to $g$ and the characteristic form $f\eta|_{T\mathcal{F}}$ is the restriction of the characteristic form of a minimal bundle-like metric on $(M,\mathcal{F})$. 

By the last paragraph, it is sufficient to show Proposition~\ref{Proposition : g_1} in the case where $\eta|_{T\mathcal{F}}$ is the restriction of the characteristic form of a minimal metric. Proposition~\ref{prop:Alvarez Lopez2} implies that there exists $\gamma$ in $A_{1,1}$ such that $\kappa = d_{0,-1}^{*}\gamma$. By Lemma~\ref{Lemma : d0-1} and the open mapping theorem, there exists a continuous inverse $\mu$ of $d_{0,-1}^{*}|_{\overline{d_{0,1}(A_{1,0})}} : \overline{d_{0,1}(A_{1,0})} \to \overline{d_{0,-1}^{*}(A_{1,1})}$. Letting $\gamma = \mu(\kappa)$, we get $\gamma$ in $A_{1,1}$ such that $\kappa$ is $C^{\infty}$-close to $0$ and $\kappa = d_{0,-1}^{*}\gamma$. Proposition~\ref{prop:Alvarez Lopez2} imply that we can take a minimal bundle-like metric $g_{1}$ which is $C^{\infty}$-close to $g$. Thus we get a minimal bundle-like metric $g_{1}$ which is $C^{\infty}$-close to $g$. Lemma~\ref{Lemma : minimal and Killing} implies that $g_{1}$ is Killing.
\end{proof}

\subsection{Proof of Theorem~\ref{Theorem : Stability}}

We recall a result in Nozawa~\cite{Nozawa 2}, which is essential in the following proof of Theorem~\ref{Theorem : Stability}:
\begin{theorem}\label{Theorem : Invariance} 
Let $V$ be a connected open set of $\mathbb{R}^{\ell}$ and $\{\mathcal{F}^{t}\}_{t \in V}$ be a smooth family of Riemannian flows on a closed manifold. Then one of the following two cases occurs:
\begin{enumerate}
\item $\mathcal{F}^{t}$ is geometrically taut for every $t$ in $V$. 
\item $\mathcal{F}^{t}$ is not geometrically taut for every $t$ in $V$. 
\end{enumerate}
\end{theorem}

\begin{proof}[Proof of Theorem~\ref{Theorem : Stability}]
We extend the smooth family of transverse metrics $\{g_{\nu}^{t}\}_{t \in V}$ to a smooth family of Riemannian metrics $\{\widehat{g}^{t}\}_{t \in V}$ on~$M$ so that $\widehat{g}^{t}$ is bundle-like with respect to $\mathcal{F}^{t}$ and $\widehat{g}^{0}={g}$. Here $(M,\mathcal{F}^{t})$ is geometrically taut by Theorem~\ref{Theorem : Invariance}. Let $V_{1}$ be a sufficiently small open neighborhood of~$0$ in $V$. By Proposition~\ref{Proposition : g_1}, for $t$ in $V_{1}$, there exists a minimal metric $\widehat{g}_{1}^{t}$ on $(M,\mathcal{F}^{t})$ for $t$ in $V_{1}$ such that $\widehat{g}_{1}^{t}$ is sufficiently close to $\widehat{g}^{t}$.

We denote the formal adjoint operator of $d$ on $(M,\widehat{g}^{t})$ by $d^{* t}$. We will show that there exists a~$1$-form $\zeta^{t}$ such that 
\begin{description}
\item[(I)] $\zeta^{t}$ is sufficiently close to $\eta^{0}$,
\item[(II)] $d^{* t} \zeta^{t}\! =\! 0$ and
\item[(III)] $d\zeta^{t}$ is a basic~$(1,1)$-form on $(M,\mathcal{F}^{t})$. 
\end{description}
Let $\tau^{t}$ be the characteristic form of $(M,\mathcal{F}^{t},\widehat{g}_{1}^{t})$, which is close to $\eta^{0}$. Here $d\tau^{t}$ represents the basic Euler class of $(M,\mathcal{F}^{t})$ by the definition. Since the~$(0,2)$-component of the basic Euler class is trivial by the assumption, there exists a basic $(0,1)$-form $\sigma^{t}$ such that $\overline{\partial} \sigma^{t}\! =\! (d\tau^{t})^{0,2}$. We consider $\sigma^{t} + \overline{\sigma^{t}}$ as a real~$1$-form on~$M$. Let $\theta^{t}\! =\! \tau^{t} - (\sigma^{t} + \overline{\sigma^{t}})$. Then $d\theta^{t}$ is a basic~$(1,1)$-form. We consider a decomposition $\theta^{t}\! =\! \theta_{1}^{t} + \theta_{2}^{t}$, where $\theta_{1}^{t}$ is an exact~$1$-form and $\theta_{2}$ is belongs to $\ker d^{* t}$ by the Hodge decomposition of $\Omega^{1}(M)_{\mathbb{C}}$ with respect to $\widehat{g}^{t}$. Let $\zeta^{t}\! =\! \theta_{2}^{t}$. It follows that $\theta_{1}^{t}$ is close to~$0$, because $\tau^{t}$ is close to $\eta^{0}$, which is belongs to $\ker d^{* 0}$. Then the condition~(I) is satisfied. The condition~(II) is clearly satisfied by the construction. Since $d\zeta^{t}\! =\! d\theta^{t}$ is a basic~$(1,1)$-form, the condition~(III) is satisfied. 

We consider differential operators $d_{1,1,-1}^{t}$, $D^{t}$ and $\Delta_{D}^{t}$ on $(M,\mathcal{F}^{t})$ as in~\eqref{Equation : t1},~\eqref{Equation : t2},~\eqref{Equation : t3} and~\eqref{Equation : t4} using a splitting $TM = (H^{t})^{1,0} \oplus (H^{t})^{0,1} \oplus (T\mathcal{F}^{t} \otimes \mathbb{C})$ determined by $H^{t}$ and $J^{t}$. Let $\mathbf{H}^{1}_{D}(t)$ be the space of $\Delta_{D}^{t}$-harmonic~$1$-forms. Let $\mathbf{H}^{1}(t)$ be the space of harmonic~$1$-forms on $(M,\widehat{g}^{t})$. We show that $\dim \mathbf{H}^{1}_{D}(t)$ is constant with respect to $t$ on $V_{1}$. By Lemma~\ref{Lemma : H 1}, we have
\begin{equation}\label{Equation : 1}
\dim \mathbf{H}^{1}_{D}(0) = \dim H^{1}(M;\mathbb{C}) + 1.
\end{equation}
Here $D^{t}$ is self-adjoint strongly elliptic operator by Lemma~\ref{lem:D} (i). Hence Theorem~4 of Kodaira and Spencer~\cite{Kodaira Spencer} implies that $\dim \mathbf{H}^{1}_{D}(t)$ is upper semicontinuous with respect to $t$. Thus we have
\begin{equation}\label{Equation : 3}
\dim \mathbf{H}^{1}_{D}(t) \leq \dim \mathbf{H}^{1}_{D}(0)
\end{equation}
for $t$ in $V_{1}$. Here $\mathbf{H}^{1}(t) \subseteq \mathbf{H}_{D}^{1}(t)$ by Lemma~\ref{lem:D} (iv). By the previous paragraph, we have a~$1$-form $\zeta^{t}$ which satisfies the conditions~(I),~(II) and~(III) above. The condition~(II) implies $D^{* t}\zeta^{t}\! =\! 0$. The condition~(III) implies $D^{t}\zeta^{t}\! =\! 0$. Note that $\zeta^{t}$ is not $d$-closed by the condition~(I) and the nontriviality of $d\eta^{0}$. It follows that $\zeta^{t}$ is not a contained in $\mathbf{H}^{1}(t)$. Thus we have
\begin{equation}\label{Equation : 4}
\dim H^{1}(M;\mathbb{C}) + 1 \leq \dim \mathbf{H}^{1}_{D}(t).
\end{equation}
By~\eqref{Equation : 1},~\eqref{Equation : 3} and~\eqref{Equation : 4}, we have
\begin{equation*}
\dim H^{1}(M;\mathbb{C}) + 1 = \dim \mathbf{H}^{1}_{D}(t).
\end{equation*}
Thus $\dim \mathbf{H}^{1}_{D}(t)$ is constant with respect to $t$.

Since $\dim \mathbf{H}^{1}_{D}(t)$ is constant with respect to $t$ by the consequence of the previous paragraph, the projection $F^{t} \colon \Omega^{1}(M)_{\mathbb{C}} \to \mathbf{H}^{1}_{D}(t)$ maps a smooth family of~$1$-forms to a smooth family of~$1$-forms by Theorem~5 of Kodaira and Spencer~\cite{Kodaira Spencer}. Letting $\varpi^{t}\! =\! F^{t}(\eta^{t})$, we have a smooth family $\{\varpi^{t}\}_{t \in V_{1}}$ of $\Delta_{D}^{t}$-harmonic~$1$-forms such that $\varpi^{0}\! =\! F^{0}(\eta^{0})\! =\! \eta^{0}$. We show that $d\varpi^{t}$ is a basic~$(1,1)$-form on $(M,\mathcal{F}^{t})$. By $D^{t}\varpi^{t}\! =\! 0$, we have $d\varpi^{t}\! =\! d_{1,1,-1}^{t}\varpi^{t}$. Thus $\iota_{\xi^{t}}d\varpi^{t}\! =\! 0$. It follows that $d\varpi^{t}$ is basic. Hence $d\varpi^{t}$ is a basic~$(1,1)$-form on $(M,\mathcal{F}^{t})$. Letting $\Real \varpi^{t}\! =\! \frac{\varpi^{t} + \overline{\varpi}^{t}}{2}$, we have a smooth family $\Real \varpi^{t}$ of real~$1$-forms such that $d\Real \varpi^{t}$ is a basic~$(1,1)$-forms. Since $d\eta^{0}$ is nondegenerate on $TM/T\mathcal{F}^{0}$, $d \Real \varpi^{t}$ is also nondegenerate  on $TM/T\mathcal{F}^{t}$ for $t$ in $V_{1}$. Thus $J^{t}$ and $d \Real \varpi^{t}$ determines a transverse metric on $TM/T\mathcal{F}^{t}$. By Proposition~\ref{Proposition : Transversely Kahler}, a pair of a transversely K\"{a}hler flow $(\mathcal{F}^{t},J^{t})$ and a contact form $\Real \varpi^{t}$ whose transverse K\"{a}hler form is $d\Real \varpi^{t}$ determines a Sasakian metric on~$M$.
\end{proof}

\begin{remark}
We remark that we can show the existence of a Sasakian metric on~$M$ which is compatible with $(\mathcal{F}^{t},J^{t})$ for $t$ close to~$0$ directly by the triviality of the~$(0,2)$-component of the basic Euler class without considering the Laplacian of~$D$. But we need the operator $D$ to show the existence of a smooth family of compatible Sasakian metrics, which has certain advantage comparing with the existence of compatible Sasakian metrics at each parameter.
\end{remark}

\section{Stability of K-contact structures}\label{sec : Kcont}

Let~$M$ be a closed manifold with a Riemannian foliation $\mathcal{F}$ and a bundle-like metric $g$. In this Section~\ref{sec : Kcont} we consider only real differential forms. Let $Q=(T\mathcal{F})^{\perp}$. We use the well known double grading $\Omega^{\bullet}(M)=\oplus_{j,k} A_{j,k}$ of the real de Rham complex $\Omega^{\bullet}(M)$, where $A_{j,k}=C^{\infty}(\wedge^{j} Q^{*} \otimes \wedge^{k} T^{*}\mathcal{F})$ instead of the triple grading considered in Section~\ref{sec:D}. Then $d$ is decomposed in terms of the double grading as $d = d_{1,0} + d_{0,1} + d_{2,-1}$ as~\eqref{eq:decompositionofd}. Let $\widehat{D} = d_{1,0} + d_{0,1}$. Since $d - \widehat{D} = d_{2,-1}$ is a differential operator of degree~$0$ by Lemma~1.1 of \'{A}lvarez L\'{o}pez~\cite{Alvarez Lopez}, it follows that $\Delta_{\widehat{D}}$ is a self-adjoint strongly elliptic operator with the same symbol as the Laplacian of $d$. Let $\mathbf{H}^{1}_{\widehat{D}}$ be the space of $\Delta_{\widehat{D}}$-harmonic~$1$-forms. An argument similar to the proof of Lemma~\ref{Lemma : H 1} shows
\begin{lemma}\label{Lemma : H 1 2}
If $(M,\mathcal{F})$ admits a $K$-contact structure whose Reeb flow is $\mathcal{F}$, then we have an equality
\begin{equation*}
\mathbf{H}^{1}_{\widehat{D}} = \mathbf{H}^{1} \oplus \mathbb{R} \eta
\end{equation*}
in $\Omega^{1}(M)$.
\end{lemma} 

Theorem~\ref{Theorem : K-contact} is proved by an argument similar to the proof of Theorem~\ref{Theorem : Stability}. We only describe the outline as follows:
\begin{proof}[Outline of the Proof of Theorem~\ref{Theorem : K-contact}]
Let $\{\widehat{g}^{t}\}_{t \in V}$ be a smooth family of Riemannian metrics on~$M$ such that $\widehat{g}^{0} = g^{0}$ and the transverse metric of $\mathcal{F}^{t}$ induced from $\widehat{g}^{t}$ is $g^{t}_{\nu}$. Theorem~\ref{Theorem : Invariance} and Proposition~\ref{Proposition : g_1} imply that there exists a Killing metric $\widehat{g}_{1}^{t}$ of $(M,\mathcal{F}^{t})$ sufficiently close to $g^{t}$ for $t$ in a sufficiently small neighborhood $V_{1}$ of~$0$ in $V$.

Let $Q^{t} = (T\mathcal{F}^{t})^{\perp}$. We consider a double grading $\Omega^{\bullet}(M) = \oplus_{j,k} A_{j,k}^{t}$ of $\Omega^{\bullet}(M)$, where $A_{j,k}^{t}=C^{\infty}(\wedge^{j} Q^{t *} \otimes \wedge^{k} T^{*}\mathcal{F}^{t})$. We decompose $d$ by double grading on $(M,\mathcal{F}^{t})$ and let $\widehat{D}^{t} = d_{1,0}^{t} + d_{0,1}^{t}$. As the second paragraph of the proof of Theorem~\ref{Theorem : Stability}, we can show that there exists a~$1$-form $\zeta^{t}$ on~$M$ such that 
\begin{description}
\item[(I')] $\zeta^{t}$ is sufficiently close to $\eta^{0}$ and
\item[(II')] $d^{* t} \zeta^{t} = 0$,
\end{description}
where $d^{* t}$is the formal adjoint of $d$ on $(M,\widehat{g}^{t})$.

We see that $\zeta^{t}$ is a $\widehat{D}^{t}$-harmonic form which is linearly independent of harmonic~$1$-forms by an argument similar to the third paragraph of the proof of Theorem~\ref{Theorem : Stability}. Lemma~\ref{Lemma : H 1 2} implies $\dim \mathbf{H}^{1}_{D}(t) = \dim \mathbf{H}^{1}(M;\mathbb{R}) + 1$, where $\mathbf{H}^{1}_{D}(t)$ is the space of $\Delta_{\widehat{D}^{t}}$-harmonic~$1$-forms.

Since $\dim \mathbf{H}^{1}_{D}(t)$ is constant with respect to $t$, by Theorem~5 of Kodaira and Spencer~\cite{Kodaira Spencer}, we get a smooth family of~$1$-forms $\varpi^{t}$ by projecting $\eta^{t}$ to $\mathbf{H}^{1}_{D}(t)$. Then $d\varpi^{t}$ is basic. Since $\varpi^{0} = \eta^{0}$ is a contact form, $\varpi^{t}$ is contact for $t$ in $V_{1}$. Thus $\varpi^{t}$ is a contact form whose Reeb flow is a Riemannian flow $\mathcal{F}^{t}$. 
\end{proof}

\section{Kodaira-Akizuki-Nakano vanishing theorem for transversely K\"{a}hler foliations}\label{Section : KAN}

\subsection{Basic Hodge star operator and basic Lefschetz operator}

We recall fundamental notion of Lefschetz theory for basic cohomology of transversely K\"{a}hler foliation introduced by El~Kacimi~Alaoui~\cite{El Kacimi Alaoui}. Let $(\mathcal{F},J,g_{\nu})$ be a complex codimension $n$ transversely K\"{a}hler foliation with transverse K\"{a}hler form $\omega$. We consider complex differential forms $\Omega^{\bullet}(M)_{\mathbb{C}} = \Omega^{\bullet}(M) \otimes \mathbb{C}$ throughout Section~\ref{Section : KAN}. At each point $x$ on~$M$, we have the Hodge star operator
\begin{equation*}
*_{b,x} \colon \wedge^{k} (T^{*}_{x}M/T^{*}_{x}\mathcal{F}) \otimes \mathbb{C} \longrightarrow \wedge^{2n-k} (T^{*}_{x}M/T^{*}_{x}\mathcal{F}) \otimes \mathbb{C} 
\end{equation*}
determined by the transverse metric $g$ on $T_{x}M/T_{x}\mathcal{F}$. A basic $k$-form on $(M,\mathcal{F})$ can be regarded as a section of $\wedge^{k} (T^{*}M/T^{*}\mathcal{F})$. Thus we have the basic Hodge operator
\begin{equation*}
*_{b} \colon \Omega_{b}^{k}(M/\mathcal{F})_{\mathbb{C}} \longrightarrow \Omega_{b}^{2n-k}(M/\mathcal{F})_{\mathbb{C}} 
\end{equation*}
Composing $*_{b}$ with complex conjugation, we have
\begin{equation*}
\overline{*}_{b} \colon \Omega_{b}^{p,q}(M/\mathcal{F})  \longrightarrow \Omega_{b}^{n-p,n-q}(M/\mathcal{F}).
\end{equation*}
Let $(E,h_{E})$ be a holomorphic $\mathcal{F}$-fibered Hermitian line bundle over $(M,\mathcal{F})$. Let
\begin{equation*}
h \colon E \longrightarrow E^{*}
\end{equation*}
be a $\mathbb{C}$-antilinear isomorphism defined by $h(s) = h_{E}(s,\cdot)$. The complex conjugation of the basic Hodge star operator on the basic Dolbeault complex with values in $E$
\begin{equation*}
\overline{*}_{b,E} \colon \Omega_{b}^{p,q}(M/\mathcal{F},E) \longrightarrow \Omega_{b}^{n-p,n-q}(M/\mathcal{F},E^{*})
\end{equation*}
is defined by $\overline{*}_{b,E}(\alpha \otimes s)=\overline{*}_{b}\alpha \otimes h(s)$ for sections of the form $\alpha \otimes s$, where $\alpha$ is a basic $(p,q)$-form and $s$ is a local holomorphic section of $E$. We define the basic Lefschetz operator
\begin{equation*}
L \colon \Omega_{b}^{k}(M/\mathcal{F},E)_{\mathbb{C}}  \longrightarrow \Omega_{b}^{k+2}(M/\mathcal{F},E)_{\mathbb{C}}
\end{equation*}
by $L \alpha = \alpha \wedge \omega$, where $\omega$ is the transverse K\"{a}hler form of $(M,\mathcal{F},J,g_{\nu})$.

\subsection{Chern connections of $\mathcal{F}$-fibered Hermitian holomorphic line bundles}
Let $(E,h_{E})$ be a holomorphic $\mathcal{F}$-fibered Hermitian line bundle over $(M,\mathcal{F})$. Let 
\begin{equation*}
\nabla_{E} \colon \Omega_{b}^{k}(M/\mathcal{F},E)_{\mathbb{C}}  \longrightarrow \Omega_{b}^{k+1}(M/\mathcal{F},E)_{\mathbb{C}}  
\end{equation*}
be a basic connection of $(E,h_{E})$. Recall that
\begin{equation*}
\nabla_{E}(\Omega_{b}^{p,q}(M/\mathcal{F},E)) \subseteq \Omega_{b}^{p+1,q}(M/\mathcal{F},E) \oplus \Omega_{b}^{p,q+1}(M/\mathcal{F},E).
\end{equation*}
We define the $(1,0)$-component $\nabla_{E}'$ and the $(0,1)$-component $\nabla_{E}''$ of $\nabla_{E}$ by 
\begin{equation*}
\nabla_{E}' \alpha = \pi^{p+1,q} \nabla_{E} \alpha, \,\,\,\, \nabla_{E}'' \alpha = \pi^{p,q+1} \nabla_{E} \alpha
\end{equation*}
for $\alpha$ in $\Omega_{b}^{p,q}(M/\mathcal{F},E)$, where $\pi^{i,j}$ is the projection to $\Omega_{b}^{i,j}(M/\mathcal{F},E)$ for $(i,j) = (p+1,q)$ and $(p,q+1)$, respectively. Throughout Section~\ref{Section : KAN}, we assume that $\nabla_{E}$ is the Chern connection of $(E,h_{E})$, that is, $\nabla_{E}$ is the unique connection on $E$ such that $\nabla_{E} h = 0$ and $\nabla_{E}'' = \overline{\partial}_{E}$, where $\overline{\partial}_{E}$ is the basic Dolbeault operator defined in Definition~\ref{Definition : Dol}. In particular, $(\nabla_{E}'')^{2} = 0$. 

Regarding the curvature form $\Theta(\nabla_{E})$ of $\nabla_{E}$ as a basic~$2$-form on~$(M,\mathcal{F})$ by a natural isomorphism 
\begin{equation*}
\Omega_{b}^{2}(M/\mathcal{F})_{\mathbb{C}} \otimes \End(E) \cong \Omega_{b}^{2}(M/\mathcal{F})_{\mathbb{C}} \otimes C^{\infty}_{b}(M/\mathcal{F})_{\mathbb{C}} \cong \Omega_{b}^{2}(M/\mathcal{F})_{\mathbb{C}},
\end{equation*}
we define the curvature operator
\begin{equation*}
\Theta(\nabla_{E}) \colon \Omega_{b}^{k}(M/\mathcal{F},E)_{\mathbb{C}} \longrightarrow \Omega_{b}^{k+2}(M/\mathcal{F},E)_{\mathbb{C}}
\end{equation*}
of $\nabla_{E}$ by $\Theta(\nabla_{E}) \alpha = \alpha \wedge \Theta(\nabla_{E})$.

Note that there exists a unique connection $\nabla_{E^{*}}$ on $E^{*}$ which satisfies
\begin{equation}\label{eq : dual}
d \langle s , s^{*}  \rangle = \langle \nabla_{E}s, s^{*}\rangle + \langle s, \nabla_{E^{*}} s^{*} \rangle,
\end{equation}
for $s$ in $C^{\infty}(E)$ and $s^{*}$ in $C^{\infty}(E^{*})$ where $\langle \cdot , \cdot \rangle$ is the coupling $E \otimes E^{*} \to \mathbb{C}$.

Let $\widetilde{\wedge}$ be the wedge product defined by the composite of
\begin{equation}\label{Equation : tilde wedge}
\xymatrix{ \Omega_{b}^{p,q}(M/\mathcal{F}, E) \times \Omega_{b}^{n-p,n-q}(M/\mathcal{F}, E^{*}) \ar[r]^>>>>>{\wedge} & \Omega_{b}^{n,n}(M/\mathcal{F}, E \otimes E^{*}) \ar[r] & \Omega_{b}^{n,n}(M/\mathcal{F}), }
\end{equation}
where the second map is induced by the coupling $E \otimes E^{*} \to \mathbb{C}$.

\subsection{Homological orientability of transversely K\"{a}hler foliations}

Let $(\mathcal{F},J,g_{\nu})$ be a transversely K\"{a}hler foliation of complex codimension $n$ on a closed manifold~$M$. We recall a terminology due to El~Kacimi~Alaoui~\cite{El Kacimi Alaoui}:
\begin{definition}\label{Definition : HO}
$(\mathcal{F},J,g_{\nu})$ is {\it homologically orientable} if the basic cohomology group $H_{b}^{2n}(M/\mathcal{F})$ of degree~$2n$ is nontrivial.
\end{definition}
\noindent We refer to Masa~\cite{Masa} or El~Kacimi~Alaoui and Hector~\cite{El Kacimi Alaoui Hector} for the basic cohomology of Riemannian foliations. Since the basic cohomology is determined only by the underlying foliation $\mathcal{F}$, it is independent of the transverse structures. By a theorem of Masa~\cite{Masa} in the general case, or by a theorem of Molino and Sergiescu~\cite{Molino Sergiescu} in the case of flows, $(\mathcal{F},g_{\nu})$ is homologically orientable if and only if $\mathcal{F}$ is geometrically taut (see Definition~\ref{Definition : GT}).
\begin{example}\label{Example : RS}
The Reeb flow $\mathcal{F}$ of a Sasakian manifold is isometric (see Definition~\ref{Definition : Isom}). Then it is geometrically taut by Lemma~\ref{Lemma : minimal and Killing} due to Carri\`{e}re. Thus $\mathcal{F}$ is homologically orientable by a theorem of Masa or a theorem of Molino and Sergiescu.
\end{example}

\subsection{An inner product of El~Kacimi~Alaoui--Hector}\label{sec : KH}

We assume that $(M,\mathcal{F})$ is homologically orientable in the sequel (see Definition~\ref{Definition : HO}). We will use the Molino theory, which is one of fundamental tools in the work of El~Kacimi~Alaoui and Hector of~\cite{El Kacimi Alaoui Hector} and El~Kacimi~Alaoui~\cite{El Kacimi Alaoui}. We refer to Molino~\cite{Molino} for the Molino theory. We define an inner product on $\Omega_{b}^{\bullet}(M/\mathcal{F},E)_{\mathbb{C}} $ under the assumption of homologically orientability following the argument of El~Kacimi~Alaoui and Hector in Section 4.5 of~\cite{El Kacimi Alaoui Hector} (see also El~Kacimi~Alaoui~\cite{El Kacimi Alaoui}), where they consider the case of trivial line bundles Let $\rho \colon M^{1} \to M$ be the orthonormal frame bundle of the normal bundle of $\mathcal{F}$. Let $\pi \colon M^{1} \to W$ be the basic fibration of $\mathcal{F}$. Let $m = \dim \SO(2n)$ and $X_{1}$, $X_{2}$, $\ldots$, $X_{m}$ be the vector fields on~$M$ which generate the principal $\SO(2n)$-action on $M^{1}$. Let $\theta_{i}$ be the basic $1$-form on $(M^{1},\mathcal{F}^{1})$ dual to $X_{i}$. We define an $m$-form $\chi$ on $M^{1}$ by
\begin{equation*}
\chi = \theta_{1} \wedge \theta_{2} \wedge \cdots \wedge \theta_{m}.
\end{equation*}
If $W$ is orientable, an inner product $(\cdot,\cdot)$ is defined by
\begin{equation}\label{Equation : inner product}
(\alpha_{1} ,\alpha_{2}) = \int_{W} \mathcal{I} \big( \rho^{*} (\alpha_{1} \widetilde{\wedge} \overline{*}_{b,E} \alpha_{2}) \wedge \chi \big)
\end{equation}
for $\alpha_{1}$ and $\alpha_{2}$ in $\Omega_{b}^{\bullet}(M/\mathcal{F},E)_{\mathbb{C}}$. Here $\mathcal{I} \colon \Omega_{b}^{k}(M/\mathcal{F})_{\mathbb{C}}  \to \Omega^{k-d}(W)_{\mathbb{C}}$ is the integration along fibers of $\pi$ defined under the assumption of the homologically orientability of $\mathcal{F}$ by El~Kacimi~Alaoui and Hector  (see Proposition~3.2 of~\cite{El Kacimi Alaoui Hector}). This $\mathcal{I}$ commutes with $d$ as shown there. If $W$ is not orientable, the orientation cover of $W$ can be used to define an inner product in the same equation as~\eqref{Equation : inner product}. In what follows, we also assume that the orientability of $W$ for the simplicity.
\begin{lemma}\label{Lemma : Stokes}
We have
\begin{equation}\label{Equation : stokes 1}
(\nabla_{E} \alpha_{1}, \alpha_{2}) = \big(\alpha_{1}, -(\overline{*}_{b,E^{*}})\nabla_{E^{*}} (\overline{*}_{b,E}) \alpha_{2}\big)
\end{equation}
for $\alpha_{1}$ in $\Omega_{b}^{k-1}(M/\mathcal{F},E)_{\mathbb{C}}$ and $\alpha_{2}$ in $\Omega_{b}^{k}(M/\mathcal{F},E)_{\mathbb{C}}$, where $\nabla_{E^{*}}$ is the connection dual to $\nabla_{E}$ mentioned in~\eqref{eq : dual}.
\end{lemma}
\begin{proof}
We closely follow the argument in Proposition~4.6 of El~Kacimi~Alaoui and Hector~\cite{El Kacimi Alaoui Hector} or Section~3.2.4 of El~Kacimi~Alaoui~\cite{El Kacimi Alaoui} where they consider the case of trivial bundles. Letting $\beta_{1} = \overline{*}_{b,E}\alpha_{2}$, we have
\begin{multline}\label{Equation : stokes 3}
d \big( \rho^{*} (\alpha_{1} \widetilde{\wedge} \beta_{1}) \wedge \chi \big) \\ 
= \big( \rho^{*} ( \nabla_{E}\alpha_{1} \widetilde{\wedge} \beta_{1}) \wedge \chi \big) + (-1)^{k-1} \big( \rho^{*} (\alpha_{1} \widetilde{\wedge} \nabla_{E^{*}}\beta_{1}) \wedge \chi \big) + (-1)^{2n-1} \big( \rho^{*} (\alpha_{1} \widetilde{\wedge} \beta_{1}) \wedge d\chi \big),
\end{multline}
where $2n$ is the real codimension of $\mathcal{F}$. Recall that $\rho \colon M^{1} \to M$ is the projection. Let $A_{j,k}(M^{1})=C^{\infty}(\wedge^{j} (\ker \chi)^{*} \otimes \wedge^{k} (\ker T\rho)^{*})$, where $T\rho$ is the differential map of $\rho$. Then we get a double grading $\Omega^{\bullet}(M^{1})_{\mathbb{C}} =\oplus_{j,k} A_{j,k}(M^{1})$. Here $d$ is decomposed into $d=d_{0,1}+d_{1,0}+d_{2,-1}$, where the subscripts means the grading. Recall that we let $m=\dim \SO(2n)$. Since $\chi$ is of degree~$(0,m)$, we can decompose $d\chi$ as
\begin{equation*}
d\chi = d_{0,1}\chi +d_{1,0}\chi +d_{2,-1}\chi,
\end{equation*}
where $d_{i,j}\chi$ is belongs to $A_{i,j+m}(M^{1})$. We get $d_{0,1}\chi = 0$, because $A_{\bullet,m+1}(M^{1})=0$. By the Rummler's formula (see the second formula in the proof of Proposition~1 in Rummler~\cite{Rummler} or Lemma~10.5.6 of Candel and Conlon~\cite{Candel Conlon}), we get $d_{1,0}\chi = -\chi \wedge \kappa_{\rho}$, where $\kappa_{\rho}$ is the mean curvature form of the foliation defined by the fibers of $\rho$. By the definition of $\chi$, we get $\kappa_{\rho}=0$, which implies $d_{1,0}\chi=0$. Thus $d\chi$ is of degree~$(2,m-1)$. Since $A_{2n+1,\bullet}(M^{1})=0$, we get $d\chi \wedge \alpha = 0$ for any $\alpha$ in $A_{2n-1,0}(M^{1})$, which implies that the third term of the right hand side of~\eqref{Equation : stokes 3} is zero. Composing $\mathcal{I}$ and $\int_{W}$ to both sides of~\eqref{Equation : stokes 3}, we have~\eqref{Equation : stokes 1}.
\end{proof}

\subsection{Formal adjoint operators with respect to $(\cdot,\cdot)$}

Recall that the formal adjoint $\mathcal{D}^{\star}$ of a $\mathbb{C}$-linear map $\mathcal{D} \colon \Omega_{b}^{\bullet}(M/\mathcal{F},E)_{\mathbb{C}} \to \Omega_{b}^{\bullet+k}(M/\mathcal{F},E)_{\mathbb{C}}$ is a $\mathbb{C}$-linear map $\mathcal{D}^{\star} \colon \Omega_{b}^{\bullet}(M/\mathcal{F},E)_{\mathbb{C}} \to \Omega_{b}^{\bullet-k}(M/\mathcal{F},E)_{\mathbb{C}}$ such that $(\mathcal{D}\alpha_{1},\alpha_{2}) = (\alpha_{1},\mathcal{D}^{\star}\alpha_{2})$ for $\alpha_{1}$ in $\Omega_{b}^{\bullet}(M/\mathcal{F},E)_{\mathbb{C}}$ and $\alpha_{2}$ in $\Omega_{b}^{\bullet+k}(M/\mathcal{F},E)_{\mathbb{C}}$. The following lemma shows that the formal adjoint operators with respect to $(\cdot,\cdot)$ are given by the conjugation by the Hodge star operators as in the case of complex manifolds. This is an advantage of the inner product $(\cdot,\cdot)$, which allows us to apply the classical argument for complex manifolds to show vanishing theorems:
\begin{lemma}\label{Lemma : formal adjoint}
Letting $\Lambda = L^{\star}$, we get
\begin{align}
\Lambda & = (\overline{*}_{b,E^{*}}) \, L \, (\overline{*}_{b,E}), \label{eq : fa1} \\
\nabla_{E}''^{\star} & = - (\overline{*}_{b,E^{*}}) \, \nabla_{E^{*}}'' \, (\overline{*}_{b,E}), \label{eq : fa2} \\
\nabla_{E}'^{\star} & = - (\overline{*}_{b,E^{*}}) \, \nabla_{E^{*}}' \, (\overline{*}_{b,E}). \label{eq : fa3} 
\end{align}
\end{lemma}
\begin{proof}
For $\alpha_{1}$ in $\Omega_{b}^{k}(M/\mathcal{F},E)_{\mathbb{C}}$ and $\alpha_{2}$ in $\Omega_{b}^{k+2}(M/\mathcal{F},E)_{\mathbb{C}}$, we get 
\begin{equation*}
(\alpha_{1}, \Lambda \alpha_{2}) = (L \alpha_{1}, \alpha_{2}) =  (\alpha_{1} \wedge \omega, \alpha_{2}) = \big(\alpha_{1},  \overline{*}_{b,E^{*}} (\omega \wedge  \overline{*}_{b,E} \alpha_{2}) \big) = (\alpha_{1},  \overline{*}_{b,E^{*}} L \overline{*}_{b,E} \alpha_{2}) 
\end{equation*}
by definition of $(\cdot,\cdot)$. Then \eqref{eq : fa1} is proved. Take $\alpha_{3}$ in $\Omega_{b}^{p,q-1}(M/\mathcal{F},E)$ and $\alpha_{4}$ in $\Omega_{b}^{p,q}(M/\mathcal{F},E)$. We get $(\nabla_{E}'\alpha_{3}, \alpha_{4}) = 0$ and $(\alpha_{3}, \overline{*}_{b,E^{*}} \nabla_{E^{*}}' \overline{*}_{b,E} \alpha_{4}) = 0$ because of the degree. Thus Lemma~\ref{Lemma : Stokes} implies
\begin{equation*}
(\nabla_{E}''\alpha_{3}, \alpha_{4}) = (\nabla_{E}\alpha_{3}, \alpha_{4}) =  ( \alpha_{3}, -\overline{*}_{b,E^{*}} \nabla_{E^{*}} \overline{*}_{b,E} \alpha_{4}) = ( \alpha_{3}, -\overline{*}_{b,E^{*}} \nabla_{E^{*}}'' \overline{*}_{b,E} \alpha_{4}).
\end{equation*}
Thus~\eqref{eq : fa2} is proved. Take $\alpha_{5}$ in $\Omega_{b}^{p-1,q}(M/\mathcal{F},E)$ and $\alpha_{6}$ in $\Omega_{b}^{p,q}(M/\mathcal{F},E)$. Because of the degree, we get $(\nabla_{E}''\alpha_{5}, \alpha_{6}) = 0$ and $(\alpha_{5}, \overline{*}_{b,E^{*}} \nabla''_{E^{*}} \overline{*}_{b,E} \alpha_{6}) = 0$. Thus Lemma~\ref{Lemma : Stokes} implies
\begin{equation*}
(\nabla_{E}'\alpha_{5}, \alpha_{6}) = (\nabla_{E}\alpha_{5}, \alpha_{6}) =  ( \alpha_{5}, -\overline{*}_{b,E^{*}} \nabla_{E^{*}} \overline{*}_{b,E} \alpha_{6}) = ( \alpha_{5}, -\overline{*}_{b,E^{*}} \nabla'_{E^{*}} \overline{*}_{b,E} \alpha_{6}).
\end{equation*}
Thus~\eqref{eq : fa3} is proved. 
\end{proof}

\begin{remark}\label{Remark : adj}
$\Omega_{b}^{\bullet,\bullet}(M/\mathcal{F},E)$ has another natural inner product obtained by the restriction of the inner product on $\Omega^{\bullet}(M,E)_{\mathbb{C}}$. In this case, the formal adjoint of $\nabla_{E}$ may be different from $- \overline{*}_{b,E^{*}} \nabla_{E^{*}} \overline{*}_{b,E}$. The difference is given by the mean curvature form of $\mathcal{F}$ (see Proposition~3.6 of Kamber and Tondeur~\cite{Kamber Tondeur} for the case where $E$ is trivial).
\end{remark}

\subsection{Bochner-Kodaira-Nakano identity}\label{sec:BKN}

For each integer $k$, let
\begin{equation*}
\mathfrak{g}_{k} = \{ A \in \End_{\mathbb{C}-\Vect} \big(\Omega^{\bullet}_{b}(M/\mathcal{F},E)_{\mathbb{C}} \big) \, | \, A\big(\Omega^{\bullet}_{b}(M/\mathcal{F},E)_{\mathbb{C}}\big) \subseteq \Omega^{\bullet+k}_{b}(M/\mathcal{F},E)_{\mathbb{C}} \}.
\end{equation*}
It is well known that $\mathfrak{g} = \bigoplus_{k \in \mathbb{Z}} \mathfrak{g}_{k}$ has a bracket operation $[\cdot,\cdot] \colon \mathfrak{g} \times \mathfrak{g} \to \mathfrak{g}$ defined by $[A,B] = AB - (-1)^{a \cdot b} BA$ for $A$ in $\mathfrak{g}_{a}$ and $B$ in $\mathfrak{g}_{b}$. The Laplacian of $\nabla_{E}''$ and $\nabla_{E}'$ is defined by
\begin{align*}
\Delta_{E}'' & = [\nabla_{E}'',\nabla_{E}''^{\star}], & \Delta_{E}' & = [\nabla_{E}',\nabla_{E}'^{\star}].
\end{align*}
The Bochner-Kodaira-Nakano identity describes the difference of two Laplacians by the curvature of $E$. For the proof of the identity, we recall our notation on local description of basic forms on $(M,\mathcal{F})$ in Section~\ref{Section : Atlas}. Let $(M,\mathcal{F},J,g_{\nu})$ be a transversely K\"{a}hler foliation. Let $(E,h)$ be a holomorphic $\mathcal{F}$-fibered Hermitian line bundle on $(M,\mathcal{F})$. As in Section~\ref{Section : Atlas}, we take a covering $\{U_{j}\}$ of~$M$ so that $(U_{j},\mathcal{F}|_{U_{j}})$ is isomorphic to the standard transversely holomorphic foliation defined by a decomposition $\mathbb{R}^{m} \times B^{n}_{j} = \sqcup_{z \in B^{n}_{j}} \mathbb{R}^{m} \times \{z\}$, where $B^{n}_{j}$ is the unit ball of $\mathbb{C}^{n}$. Let $\phi_{j} \colon U_{j} \to B^{n}_{j}$ be the composite of $U_{j} \cong \mathbb{R}^{m} \times B^{n}_{j} \to B^{n}_{j}$, where the second map is the second projection. By definition of holomorphic $\mathcal{F}$-fibered Hermitian line bundle (see Definition~\ref{Definition : F fibered vector bundle}), there exists a Hermitian holomorphic line bundle $(E_{j},h_{j})$ and a K\"{a}hler metric $g_{j}$ on $B^{n}_{j}$ which determines $g_{\nu}$, $E$ and $h$ by
\begin{align*}
g_{\nu}|_{U_{j}} & = \phi_{j}^{*}g_{j}, \\
E|_{U_{j}} & = \phi_{j}^{*}E_{j}, \\
h|_{U_{j}} & = \phi_{j}^{*}h_{j}.
\end{align*}
We consider the Chern connection $\nabla_{j}$ of $(E_{j},h_{j})$, the Lefschetz operator $L_{j}$ on $B^{n}_{j}$ and the complex conjugate of the Hodge star operator $\overline{*}_{E_{j}}$ on $B^{n}_{j}$ with values in~$E_{j}$. We denote the curvature operator of $\nabla_{j}$ by $\Theta(\nabla_{j})$. Clearly we get
\begin{align}
(\overline{*}_{b,E}) \phi_{j}^{*} \beta & = \phi_{j}^{*} (\overline{*}_{E_{j}} \beta), \notag \\
L \phi_{j}^{*} \beta & = \phi_{j}^{*} (L_{j} \beta), \notag \\
\nabla_{E}' \phi_{j}^{*} \beta & = \phi_{j}^{*} (\nabla_{j}' \beta), \label{eq:naturality} \\
\nabla_{E}'' \phi_{j}^{*} \beta & = \phi_{j}^{*} (\nabla_{j}'' \beta), \notag \\
\Theta(\nabla_{E}) \phi_{j}^{*} \beta & = \phi_{j}^{*} (\Theta(\nabla_{j}) \notag \beta)
\end{align}
for a differential form $\beta$ on $B^{n}_{j}$ valued in $E_{j}$. We have the following Bochner-Kodaira-Nakano identity for homologically orientable transversely K\"{a}hler foliations:  
\begin{lemma}\label{Lemma : Nakano Bochner Kodaira Nakano}
We get 
\begin{equation}\label{eq : BKN}
\Delta_{E}'' = \Delta_{E}' + [\sqrt{-1}\Theta(\nabla_{E}),\Lambda], 
\end{equation}
where $\Delta_{E}'' = [\nabla_{E}'',\nabla_{E}''^{\star}]$ and $\Delta_{E}' = [\nabla_{E}',\nabla_{E}'^{\star}]$.
\end{lemma}

\begin{proof}
We consider $U_{j}$ and $\phi_{j} \colon U_{j} \to B^{n}_{j}$ mentioned above in this Section~\ref{sec:BKN}. Since the formula is local, we can consider on each $U_{j}$. By the Bochner-Kodaira-Nakano identity (see Sections~4.5 and~4.6 of Demailly~\cite{Demailly 2} or Theorem 1.1 of Chapter VII of Demailly~\cite{Demailly 3}) for the Hermitian holomorphic line bundle $E_{j}$ on $B^{n}_{j}$, we have
\begin{equation}\label{Equation : Equation on B n 1}
\Delta_{E_{j}}'' = \Delta_{E_{j}}' + [\sqrt{-1}\Theta(\nabla_{j}),\Lambda_{j}],
\end{equation}
where $\Lambda_{j} = \overline{*}_{E_{j}^{*}} L_{j} \, \overline{*}_{E_{j}}$. Note that~\eqref{Equation : Equation on B n 1} is proved only by local calculation on a complex coordinate chart, which can be applied to $E_{j}$ and $B^{n}_{j}$ in our situation. Then~\eqref{eq : BKN} is proved by pulling back~\eqref{Equation : Equation on B n 1} by $\phi_{j}$ by~\eqref{eq : fa1},~\eqref{eq : fa2},~\eqref{eq : fa3} of Lemma~\ref{Lemma : formal adjoint} and five equations in~\eqref{eq:naturality}.
\end{proof}

\subsection{Hodge decomposition for basic Dolbeault complex valued on $E$}

The Hodge decomposition theorem of El~Kacimi~Alaoui is as follows in this context:
\begin{theorem}[Theorem 2.8.7 of El~Kacimi~Alaoui~\cite{El Kacimi Alaoui}]\label{Theorem : Hodge decomposition}
There is an orthogonal decomposition
\begin{equation*}
\Omega^{p,q}_{b}(M/\mathcal{F},E) = \ker \Delta_{E}'' \oplus \Image \nabla_{E}'' \oplus \Image \nabla_{E}''^{\star}
\end{equation*}
with respect to the inner product $(\cdot,\cdot)$, where $\Delta_{E}'' = [\nabla_{E}'',\nabla_{E}''^{\star}]$.
\end{theorem}
\begin{proof}[Outline of the proof of Theorem~\ref{Theorem : Hodge decomposition}]
Let
\begin{equation*}
E^{p,q} = \wedge^{p} (T^{*}M/T^{*}\mathcal{F})^{1,0} \otimes \wedge^{q} (T^{*}M/T^{*}\mathcal{F})^{0,1} \otimes E.
\end{equation*}
Then $\Omega^{p,q}_{b}(M/\mathcal{F},E)$ is identified with the space $C^{\infty}(E^{p,q},\mathcal{F})$ of global basic sections of $E^{p,q}$ on $(M,\mathcal{F})$. Let $\rho \colon M^{1} \to M$ be the orthonormal frame bundle of $TM/T\mathcal{F}$. We denote the horizontal lift of $\mathcal{F}$ to $M^{1}$ by $\mathcal{F}^{1}$. Let $M^{1} \to W$ be the basic fibration of $\mathcal{F}$. Let $E^{1} = \rho^{*}E^{p,q}$. Here $E^{1}$ is $\mathcal{F}^{1}$-fibered. $C^{\infty}(E^{p,q},\mathcal{F})$ is identified with the space $C^{\infty}(E^{1},\mathcal{F}^{1})^{\SO(2n)}$ of $\SO(2n)$-invariant basic section of $E^{1}$. The restriction of a global basic section $s$ of $E^{1}$ to the closure $\overline{L}$ of a leaf of $\mathcal{F}^{1}$ is determined by $s(x)$ at a point $x$ in $\overline{L}$. Thus $C^{\infty}(E^{1},\mathcal{F}^{1})$ can be identified with the space $C^{\infty}_{W}(\overline{E})$ of global sections of a vector bundle $\overline{E}$ over~$W$. The inner product $(\cdot,\cdot)$ induces a Hermitian metric of $\overline{E}$. El~Kacimi~Alaoui constructed a differential operator $\mathcal{D}' \colon C^{\infty}(E^{1},\mathcal{F}^{1}) \to C^{\infty}(E^{1},\mathcal{F}^{1})$, which is identified with a self-adjoint strongly elliptic operator $\overline{\mathcal{D}} \colon C^{\infty}_{W}(\overline{E}) \to C^{\infty}_{W}(\overline{E})$ and whose restriction to $C^{\infty}(E^{1},\mathcal{F}^{1})^{\SO(2n)}$ is identified with $\Delta_{E}''$. Then the Hodge decomposition $C^{\infty}(E^{p,q},\mathcal{F}) = \ker \Delta_{E}'' \oplus \Image \Delta_{E}''$ for $\Delta_{E}''$ is obtained by the restriction of the classical Hodge decomposition $C^{\infty}(\overline{E}) = \ker \overline{\mathcal{D}} \oplus \Image \overline{\mathcal{D}}$ for $\overline{\mathcal{D}}$ to $C^{\infty}(\overline{E})^{\SO(2n)}$. The equality $\Image \Delta_{E}'' = \Image \nabla_{E}'' \oplus \Image \nabla_{E}''^{\star}$in $\Omega^{p,q}_{b}(M/\mathcal{F},E)$ follows from decompositions $\Omega^{i,j}_{b}(M/\mathcal{F},E) = \ker \Delta_{E}'' \oplus \Image \Delta_{E}''$ for $(i,j) = (p,q-1)$, $(p,q+1)$ and $\Delta_{E}'' = (\nabla_{E}'' + \nabla_{E}''^{\star})^{2}$.
\end{proof}

\subsection{Serre duality valued on $E$}

By~\eqref{Equation : stokes 1} in Lemma~\ref{Lemma : Stokes}, the pairing $\widetilde{\wedge}$ defined in~\eqref{Equation : tilde wedge} induces a product on basic Dolbeault cohomology 
\begin{equation}\label{Equation : Pairing}
H^{p,q}_{b}(M/\mathcal{F},E) \times H^{n-p,n-q}_{b}(M/\mathcal{F},E^{*}) \longrightarrow \mathbb{C}.
\end{equation}
The Serre duality follows from the Hodge decomposition Theorem~\ref{Theorem : Hodge decomposition} and the formula of formal adjoint operators in Lemma~\ref{Lemma : formal adjoint} as in the case of complex manifolds:

\begin{theorem}\label{Theorem : Serre duality}
The pairing~\eqref{Equation : Pairing} is nondegenerate, which implies $H^{p,q}_{b}(M/\mathcal{F},E) \cong H^{n-p,n-q}_{b}(M/\mathcal{F},E^{*})^{*}$.
\end{theorem}

\begin{proof}
Let $\mathbf{H}^{p,q}_{b}(E)$ (resp. $\mathbf{H}^{p,q}_{b}(E^{*})$) be the space of basic harmonic $(p,q)$-forms valued on $E$ (resp. $E^{*}$). For each $\alpha$ in $\mathbf{H}_{b}^{p,q}(E)$, we have $\int_{M} \mathcal{I}(\alpha \widetilde{\wedge} \overline{*}_{b,E}\alpha \wedge \chi) \neq 0$. Equation~\eqref{eq : fa2} of Lemma~\ref{Lemma : formal adjoint} implies $\Delta_{E}'' = \overline{*}_{b,E^{*}} \Delta_{E^{*}}'' \overline{*}_{b,E}$. Thus $\overline{*}_{b,E}\alpha \in \mathbf{H}^{n-p,n-q}_{b}(E^{*})$. Then the pairing $\mathbf{H}^{p,q}(E) \times \mathbf{H}^{n-p,n-q}(E^{*}) \to \mathbb{C}$ defined by the restriction of~\eqref{Equation : Pairing} is nondegenerate, which concludes the proof of Theorem~\ref{Theorem : Serre duality}.
\end{proof}

\subsection{$\mathcal{F}$-fibered Positive line bundles}

Let $(E,h_{E})$ be a holomorphic $\mathcal{F}$-fibered Hermitian line bundle on $(M,\mathcal{F})$. We recall the definition of the positivity of $(E,h_{E})$ due to Boyer, Galicki and Nakamaye~\cite{Boyer Galicki Nakamaye}, which is defined in a way similar to the case of line bundles on complex manifolds:
\begin{definition}\label{Definition : Positivity of E}
$E$ is called {\it positive} if there exists a Hermitian metric $h$ on $E$ such that the basic first Chern form $\frac{\sqrt{-1}}{2\pi}\Theta(\nabla_{E})$ is a transverse K\"{a}hler form of $(\mathcal{F},J)$ for a transverse metric on $(M,\mathcal{F})$.
\end{definition}
The positivity of $E$ is determined by the first Chern class as follows:
\begin{proposition}[El~Kacimi~Alaoui, Section 3.5.6 of~\cite{El Kacimi Alaoui}]\label{Proposition : c 1}
If there exists a transverse K\"{a}hler metric on $TM/T\mathcal{F}$ with transverse K\"{a}hler form $\omega$ such that $[\frac{\sqrt{-1}}{2\pi}\Theta(\nabla_{E})] = [\omega]$ in $H^{2}_{b}(M/\mathcal{F})$, then there exists a Hermitian metric $h'$ on $E$ such that $(E,\widehat{h})$ is a holomorphic $\mathcal{F}$-fibered Hermitian line bundle whose curvature form $\Theta(\widehat{\nabla}_{E})$ satisfies $\frac{\sqrt{-1}}{2\pi}\Theta(\widehat{\nabla}_{E}) = \omega$.
\end{proposition}

\subsection{Proof of vanishing theorems of Kodaira-Akizuki-Nakano, Girbau and Grauert-Riemenschneider}

\noindent Theorem~\ref{Theorem : Kodaira Akizuki Nakano vanishing theorem} follows from~\eqref{eq : BKN} of Lemma~\ref{Lemma : Nakano Bochner Kodaira Nakano}~ as in the case of complex manifolds. The proof below is essentially taken from Section~4.9 of Demailly~\cite{Demailly 2}. 

\begin{proof}[Proof of Theorem~\ref{Theorem : Kodaira Akizuki Nakano vanishing theorem}]

We use the notation in Section~\ref{sec:BKN}. Let $x$ be a point on~$M$ and $\phi_{j} : U_{j} \to B_{j}^{n}$ be a member of the defining~$1$-cocycle of $\mathcal{F}$ which contains $x$. Let $\omega_{j}$ be a K\"{a}hler form on $B_{j}^{n}$ which satisfies $\omega|_{U_{j}} = \phi_{j}^{*}\omega_{j}$, where $\omega$ is the transverse K\"{a}hler form of $(M,\mathcal{F},J,g_{\nu})$. We can modify $\phi_{j}$ so that $(\omega_{j})_{\phi_{j}(x)} = \sum_{i=1}^{n} \sqrt{-1} dz_{i} \wedge d\overline{z}_{i}$ at $\phi_{j}(x)$ by the holomorphic coordinate change of $B_{j}^{n}$. We regard the curvature form $(\sqrt{-1}\Theta(\nabla_{j}))_{\phi_{j}(x)}$ of $E_{j}$ at $\phi_{j}(x)$ as a Hermitian matrix. By diagonalizing this Hermitian matrix $(\sqrt{-1}\Theta(\nabla_{j}))_{\phi_{j}(x)}$ by a unitary matrix, we get a holomorphic coordinate $(w_{1}, w_{2}, \ldots, w_{n})$ of $B_{j}^{n}$ such that $\omega_{\phi_{j}(x)} = \sum_{i=1}^{n} \sqrt{-1} dw_{i} \wedge d\overline{w}_{i}$ and $(\sqrt{-1}\Theta(\nabla_{j}))_{\phi_{j}(x)} = \sum_{i=1}^{n} \sqrt{-1} \lambda_{i}(x) dw_{i} \wedge d\overline{w}_{i}$ for some real numbers $\lambda_{i}(x)$. The assumption of the positivity of $E$ implies that $\lambda_{i}(x)$ is positive. We reorder $\lambda_{i}(x)$ so that $\lambda_{1}(x) \leq \lambda_{2}(x) \leq \cdots \leq \lambda_{n}(x)$. We consider a vector bundle
\begin{equation*}
E^{p,q}=\wedge^{p} (T^{*}M/T^{*}\mathcal{F})^{1,0} \otimes \wedge^{q} (T^{*}M/T^{*}\mathcal{F})^{0,1} \otimes E
\end{equation*}
over $(M,\mathcal{F})$. Then $E^{p,q}$ is $\mathcal{F}$-fibered, and basic $(p,q)$-forms are regarded as basic sections of $E^{p,q}$. Let
\begin{equation*}
E^{p,q}_{j} = \wedge^{p} (T^{*}B^{n}_{j})^{1,0} \otimes \wedge^{q} (T^{*}B^{n}_{j})^{0,1} \otimes E_{j}
\end{equation*}
be a vector bundle over $B^{n}_{j}$. Then we get $E^{p,q}|_{U_{j}} = \phi_{j}^{*}E^{p,q}_{j}$. Let $g^{p,q}$ be the metric on $E^{p,q}$ induced from $g$ and $h_{E}$. Let $g^{p,q}_{j}$ be the metric on $E^{p,q}_{j}$ which satisfies $\phi_{j}^{*}g^{p,q}_{j} = g^{p,q}$. By an argument of linear algebra on $(E^{p,q}_{j})_{\phi_{j}(x)}$, we get
\begin{equation*}
g^{p,q}_{j}([\sqrt{-1}\Theta(\nabla_{j}),\Lambda_{j}] \alpha_{j}, \alpha_{j}) \geq \Big( \sum_{i=1}^{\min\{p,q\}} \lambda_{i}(x) - \sum_{i=\max\{p,q\}+1}^{n}\lambda_{i}(x) \Big) g^{p,q}_{j}(\alpha_{j},\alpha_{j})
\end{equation*}
for $\alpha_{j}$ in $(E^{p,q}_{j})_{\phi_{j}(x)}$. Then, for a basic $(p,q)$-form $\alpha$ valued in $E$, we get
\begin{equation}\label{eq : ineq}
g^{p,q}_{x}([\sqrt{-1}\Theta(\nabla_{E}),\Lambda] \alpha_{x}, \alpha_{x}) \geq \Big( \sum_{i=1}^{\min\{p,q\}} \lambda_{i}(x) -  \sum_{i=\max\{p,q\}+1}^{n}\lambda_{i}(x) \Big) g^{p,q}_{x}(\alpha_{x},\alpha_{x})
\end{equation}
at each point $x$ on~$M$. Note that $\lambda_{j}(x)$ is a leafwise constant continuous function on~$M$. Since $\alpha_{x} \widetilde{\wedge} \overline{*}_{b,E} \beta_{x} = g^{p,q}_{x}(\alpha_{x},\beta_{x}) \omega^{n}$ at each point $x$ on~$M$ for $\alpha$ and $\beta$ in $\Omega_{b}^{p,q}(M/\mathcal{F},E)$, we get  
\begin{equation*}
([\sqrt{-1}\Theta(\nabla_{E}),\Lambda] \alpha, \alpha) \geq \int_{W} \Big( \sum_{i=1}^{q} \lambda_{i} - \sum_{i=p+1}^{n}\lambda_{i}  \Big) \mathcal{I}(\rho^{*}(\alpha \widetilde{\wedge} \overline{*}_{b,E} \alpha ) \wedge \chi )
\end{equation*}
by~\eqref{eq : ineq} (see Section~\ref{sec : KH} for the definition of the inner product $(\cdot,\cdot)$ and the right hand side). Since $(\Delta''_{E}\alpha,\alpha) = ||\nabla_{E}''\alpha||^{2} + ||\nabla_{E}''^{\star}\alpha||^{2}$ and $(\Delta'_{E}\alpha,\alpha) \geq 0$, the Bochner-Kodaira-Nakano identity~\eqref{eq : BKN} of Lemma~\ref{Lemma : Nakano Bochner Kodaira Nakano} (ii) implies
\begin{equation*}
||\nabla_{E}''\alpha||^{2} + ||\nabla_{E}''^{\star}\alpha||^{2} \geq ([\sqrt{-1}\Theta(\nabla_{E}),\Lambda] \alpha, \alpha).
\end{equation*}
By Proposition~\ref{Proposition : c 1}, we fix a transverse K\"{a}hler metric on $(M,\mathcal{F},J)$ whose transverse K\"{a}hler form is equal to $\sqrt{-1}\Theta(\nabla_{E})$. Then $\lambda_{j}(x)=1$ for every $j$ and $x$. We get
\begin{equation*}
||\nabla_{E}''\alpha||^{2} + ||\nabla_{E}''^{\star}\alpha||^{2} \geq (p+q - n) ||\alpha||^{2}.
\end{equation*}
Thus if $\alpha$ is harmonic and $p + q > n$, we get $\alpha=0$.
\end{proof}

The analog of Grauert-Riemenschneider's vanishing theorem~\cite{Grauert Riemenschneider} (Theorem~\ref{Theorem : GR}) and Girbau's vanishing theorem~\cite{Girbau} (Theorem~\ref{Theorem : G})  also follow from the Bochner-Kodaira-Nakano equality as in~\cite{Demailly}: Let $(M,\mathcal{F})$ be a closed manifold with a homologically orientable transversely K\"{a}hler foliation. Let $(E,h_{E})$ be a holomorphic $\mathcal{F}$-fibered Hermitian line bundle on~$(M,\mathcal{F})$. The rank of $\Theta(\nabla_{E})$ at a point $x$ on $M$ is defined by the number of positive eigenvalues as a Hermitian matrix.
\begin{theorem}\label{Theorem : GR}
If $\Theta(\nabla_{E})$ has rank at least equal to $s$ at some point on $M$, then 
\begin{equation*}
H^{0,q}_{b}(M/\mathcal{F},E^{*}) = 0
\end{equation*}
for $q < s$.
\end{theorem}
\begin{theorem}\label{Theorem : G} 
If $\Theta(\nabla_{E})$ has rank at least equal to $s$ at every point on $M$, then
\begin{equation*}
H^{p,q}_{b}(M/\mathcal{F},E^{*}) = 0
\end{equation*}
for $p+q < s$.
\end{theorem}
\noindent We omit the proof of Theorems~\ref{Theorem : GR} and~\ref{Theorem : G}, which are similar to the argument of Demailly in the proof of Proposition in Section~2 of~\cite{Demailly} .

\begin{remark}
We remark on the other possibility of the inner product on the basic de Rham complex. \'{A}lvarez L\'{o}pez~\cite{Alvarez Lopez} proved the Hodge decomposition theorem for basic de Rham complex with respect to the inner product $\langle \langle \cdot,\cdot  \rangle \rangle$ defined by the restriction of the usual inner product of the de Rham complex. Note that the formal adjoint of $d$ with respect to $\langle \langle \cdot,\cdot \rangle \rangle$ is not given by the basic Hodge star operator as remarked in Remark~\ref{Remark : adj}. Hence Lemma~\ref{Lemma : formal adjoint} is not true in general for $\langle \langle \cdot,\cdot \rangle \rangle$. To show Theorem~\ref{Theorem : Kodaira Akizuki Nakano vanishing theorem} using $\langle \langle \cdot,\cdot \rangle \rangle$, one can apply Masa's theorem in~\cite{Masa} for the existence of a minimal metric on homologically orientable foliations. Since the mean curvature form is zero for a minimal metric, the adjoint of $d$ with respect to $\langle \langle \cdot,\cdot \rangle \rangle$ is given by its conjugation of $d$ by the basic Hodge operator. Then Lemma~\ref{Lemma : formal adjoint} is true for $\langle \langle \cdot,\cdot \rangle \rangle$. Then the rest of the argument is the same as above.
\end{remark}

\subsection{Proof of vanishing Theorem for positive Sasakian manifolds}\label{Section : pos}

We will prove Theorem~\ref{Theorem : Vanishing} by Theorem~\ref{Theorem : Kodaira Akizuki Nakano vanishing theorem}. Let $(\mathcal{F},J,g_{\nu})$ be a complex codimension $n$ transversely K\"{a}hler foliation on a closed manifold~$M$. We recall
\begin{definition}
A holomorphic $\mathcal{F}$-fibered Hermitian line bundle
\begin{equation*}
K_{\mathcal{F}} = \wedge^{n} (T^{*}M/T^{*}\mathcal{F})^{1,0}
\end{equation*}
on~$M$ is called the {\it canonical line bundle} of $(M,\mathcal{F},J)$. The dual of the canonical line bundle of $(M,\mathcal{F},J)$ is called the {\it anticanonical line bundle} of $(M,\mathcal{F},J)$. A Sasakian manifold is called {\it positive} if the anticanonical line bundle of the underlying transversely K\"{a}hler flow is positive.
\end{definition}
\noindent We refer to Boyer, Galicki and Nakamaye~\cite{Boyer Galicki Nakamaye} for more detailed information on positive Sasakian manifolds.

We assume that the homologically orientability of $\mathcal{F}$ to apply the results in this Section~\ref{Section : KAN}. Let $K_{\mathcal{F}}$ be the canonical line bundle of $(M,\mathcal{F})$. We fix a transverse Hermitian metric on $(TM/T\mathcal{F}) \otimes \mathbb{C}$. Let $h_{K_{\mathcal{F}}}$ be the Hermitian metric on $K_{\mathcal{F}}$ induced from the metric on $(TM/T\mathcal{F}) \otimes \mathbb{C}$. We recall a $\mathbb{C}$-antilinear isomorphism
\begin{equation*}
h \colon K_{\mathcal{F}} \longrightarrow K_{\mathcal{F}}^{*}
\end{equation*}
defined by $h(s) = h_{K_{\mathcal{F}}}(s,\cdot)$. Note that the canonical line bundle $(K_{\mathcal{F}},h_{K_{\mathcal{F}}})$ of $\mathcal{F}$ is a holomorphic  $\mathcal{F}$-fibered Hermitian line bundle.

Let $\mathbf{1}_{K_{\mathcal{F}}}$ be the identity in $\End(K_{\mathcal{F}}) = C^{\infty}(K_{\mathcal{F}} \otimes K_{\mathcal{F}}^{*})$. We have natural maps
\begin{align*}
\Xi_{1} \colon \Omega_{b}^{0,q}(M/\mathcal{F}) \longrightarrow \Omega_{b}^{n,q}(M/\mathcal{F},K_{\mathcal{F}}^{*}), \\ 
\Xi_{2} \colon \Omega_{b}^{n,q}(M/\mathcal{F}) \longrightarrow \Omega_{b}^{0,q}(M/\mathcal{F},K_{\mathcal{F}})
\end{align*}
defined by $\Xi_{1}(\alpha) = \alpha \wedge \mathbf{1}_{K_{\mathcal{F}}} $ and $\Xi_{2}(\beta \wedge \alpha) = \alpha \otimes \beta$ for $\beta$ in $\Omega_{b}^{0,q}(M/\mathcal{F})$ and $\beta$ in $\Omega_{b}^{n,0}(M/\mathcal{F}) = C^{\infty}(K_{\mathcal{F}})$. Clearly $\Xi_{1}$ and $\Xi_{2}$ are bijective.

Let $\mathbf{H}_{b}^{p,q}$ be the space of basic harmonic $(p,q)$-forms. Let $\mathbf{H}_{b}^{p,q}(K_{\mathcal{F}}^{*})$ be the space of basic harmonic $(p,q)$-forms with values in $K_{\mathcal{F}}^{*}$.
\begin{lemma}\label{Lemma : duality}
\begin{equation*}
\Xi_{1}(\mathbf{H}_{b}^{0,q}) = \mathbf{H}_{b}^{n,q}(K_{\mathcal{F}}^{*}). \label{Equation : duality 2}
\end{equation*}
\end{lemma}

\begin{proof}
It is easy to see
\begin{align}
\label{Equation : theta and del} \Xi_{1} \overline{\partial} = \nabla_{K_{\mathcal{F}}^{*}}'' \Xi_{1}, \\
\label{Equation : xi and del} \Xi_{2} \overline{\partial} = \nabla_{K_{\mathcal{F}}}'' \Xi_{2}, 
\end{align}
where~\eqref{Equation : theta and del} follows from the fact that $\mathbf{1}_{K_{\mathcal{F}}}$ is a holomorphic section of $K_{\mathcal{F}} \otimes K_{\mathcal{F}}^{*}$. By~\eqref{Equation : theta and del} and the bijectivity of $\Xi_{1}$, we get 
\begin{equation}\label{Equation : T1}
\Xi_{1} (\ker \overline{\partial}) = \ker \nabla_{K_{\mathcal{F}}^{*}}''.
\end{equation}

Let $s$ be a local basic section of $K_{\mathcal{F}}$ such that $h(s,s)=1$. It is easy to see 
\begin{equation*}
s \otimes h(s) = \mathbf{1}_{K_{\mathcal{F}}}.
\end{equation*}
By definition of $\overline{*}_{b}$ and $h(s,s)=1$, we have
\begin{equation*}
\overline{*}_{b} \alpha = (-1)^{nq} s \wedge \overline{*}_{b} (s \wedge \alpha).
\end{equation*}
Thus we get
\begin{align}
\textstyle \overline{*}_{b,K_{\mathcal{F}}^{*}} \Xi_{1} (\alpha) 
& \textstyle = \overline{*}_{b,K_{\mathcal{F}}^{*}} \big( \alpha \wedge \mathbf{1}_{K_{\mathcal{F}}} \big) \notag \\
& \textstyle = \overline{*}_{b,K_{\mathcal{F}}^{*}} \big( \alpha \wedge s \otimes h(s) \big) \notag \\
& \textstyle = \overline{*}_{b} (\alpha \wedge s) \otimes s \label{Equation : * and theta} \\
& \textstyle = \Xi_{2} \big( s \wedge \overline{*}_{b} (\alpha \wedge s) \big) \notag \\
& \textstyle = \Xi_{2} \overline{*}_{b} \alpha \notag.
\end{align}
By substituting $\overline{*}_{b} \beta$ to $\alpha$ in~\eqref{Equation : * and theta} and composing $\overline{*}_{b,K_{\mathcal{F}}}$ to both sides of the equation, we get 
\begin{equation}\label{Equation : * and theta 2}
\Xi_{1} \overline{*}_{b} = \overline{*}_{b,K_{\mathcal{F}}} \Xi_{2}.
\end{equation}
By~\eqref{eq : fa2} of Lemma~\ref{Lemma : formal adjoint}, we get $\overline{\partial}^{\star}  = -\overline{*}_{b} \overline{\partial} \overline{*}_{b}$. Hence, by equations~\eqref{Equation : xi and del},~\eqref{Equation : * and theta} and~\eqref{Equation : * and theta 2}, we get
\begin{equation}\label{Equation : com*}
\Xi_{1} \overline{\partial}^{\star} = - \Xi_{1} \overline{*}_{b} \overline{\partial} \overline{*}_{b} = - \overline{*}_{b,K_{\mathcal{F}}} \nabla_{K_{\mathcal{F}}}'' \overline{*}_{b,K_{\mathcal{F}}^{*}} \Xi_{1} = \nabla_{K_{\mathcal{F}}^{*}}''^{\star} \Xi_{1}.
\end{equation}

By~\eqref{Equation : com*} and the bijectivity of $\Xi_{1}$, we get
\begin{equation}\label{Equation : T2}
\Xi_{1} (\ker \overline{\partial}^{\star}) = \ker \nabla_{K_{\mathcal{F}}^{*}}''^{\star}.
\end{equation}
Since $\mathbf{H}_{b}^{0,q} = \ker \overline{\partial} \cap \ker \overline{\partial}^{\star}$ and $\mathbf{H}_{b}^{n,q}(K_{\mathcal{F}}^{*}) = \ker \nabla_{K_{\mathcal{F}}^{*}}'' \cap \ker \nabla_{K_{\mathcal{F}}^{*}}''^{\star}$ by the Hodge decomposition theorem of El~Kacimi~Alaoui (see Theorem~\ref{Theorem : Hodge decomposition}), the proof is done by~\eqref{Equation : T1} and~\eqref{Equation : T2}.
\end{proof}

\begin{remark}
We can prove $\Xi_{2}(\mathbf{H}_{b}^{n,q}) = \mathbf{H}_{b}^{0,q}(K_{\mathcal{F}})$ similarly.
\end{remark}

Lemma~\ref{Lemma : duality} allows us to prove Theorem~\ref{Theorem : Vanishing} by Theorem~\ref{Theorem : Kodaira Akizuki Nakano vanishing theorem} without using the sheaf cohomology of basic holomorphic forms.
\begin{proof}[Proof of Theorem~\ref{Theorem : Vanishing}]
It is known that The Reeb flow of a Sasakian manifold is homologically orientable as we saw in Example~\ref{Example : RS}. Let $\mathbf{H}_{b}^{p,q}$ be the space of $\Delta_{\overline{\partial}}$-harmonic basic $(p,q)$-forms on $(M,\mathcal{F})$. By the K\"{a}hler identity for homologically orientable transversely K\"{a}hler foliations by El~Kacimi~Alaoui (Proposition~3.4.5 of~\cite{El Kacimi Alaoui}), $\mathbf{H}_{b}^{p,q}$ is equal to the space of $\Delta_{d}$-harmonic $(p,q)$-forms, where $\Delta_{d} = [d,d^{\star}]$. Since $\Delta_{d}$ commutes with the complex conjugation, we get
\begin{equation*}
\mathbf{H}_{b}^{p,0} \cong \mathbf{H}_{b}^{0,p}
\end{equation*}
by the complex conjugation. By Lemma~\ref{Lemma : duality}, we get
\begin{equation*}
\mathbf{H}_{b}^{0,p} \cong \mathbf{H}_{b}^{n,p}(K_{\mathcal{F}}^{*}),
\end{equation*}
where $\mathbf{H}_{b}^{n,p}(K_{\mathcal{F}}^{*})$ is the space of $\Delta''_{K_{\mathcal{F}}^{*}}$-harmonic basic $(n,p)$-forms on $(M,\mathcal{F})$ valued in $K_{\mathcal{F}}^{*}$. Since $K_{\mathcal{F}}^{*}$ is positive, Theorem~\ref{Theorem : Kodaira Akizuki Nakano vanishing theorem} implies that $\mathbf{H}_{b}^{n,p}(K_{\mathcal{F}}^{*})$ vanishes for $p > 0$.
\end{proof}

\section{Moduli space of Sasakian metrics with a fixed underlying transverse K\"{a}hler flow}\label{Section : Moduli}

We will prove Theorem~\ref{Theorem : H1}.

Let~$M$ be a closed manifold with a Sasakian metric with contact form $\eta$. Let $\mathfrak{k}_{0}$ be the underlying transversely K\"{a}hler flow. Recall that $\Diff_{0}(\mathfrak{k}_{0})$ is the identity component of the subgroup of $\Diff(M)$ consisting of diffeomorphisms which preserve $\mathfrak{k}_{0}$, and $\Ham(\mathfrak{k}_{0})$ is defined in~\eqref{Equation : Ham}. Obviously we have an exact sequence
\begin{equation*}
\xymatrix{ 0 \ar[r] & \Ham(\mathfrak{k}_{0}) \ar[r] & \Diff_{0}(\mathfrak{k}_{0}) \ar[r]^{\Pi} & H^{1}(M;\mathbb{R}),}
\end{equation*}
where $\Pi$ is defined by $\Pi(f) = [\eta-f^{*}\eta]$ for $f$ in $\Diff_{0}(\mathfrak{k}_{0})$. Note that $\eta -f^{*}\eta$ is closed, because $f$ preserves the transverse K\"{a}hler form $d\eta$.

Recall that we denote the subgroup of $\Diff(M)$ consisting of diffeomorphisms which map each leaf of $\mathcal{F}$ to itself by $\Diff(M,\mathcal{F})$. The identity component of  $\Diff(M,\mathcal{F})$ with respect to the Fr\'{e}chet topology is denoted by $\Diff_{0}(M,\mathcal{F})$. Note that $\Diff_{0}(M,\mathcal{F})$ is a subgroup of $\Ham(\mathfrak{k}_{0})$.
\begin{lemma}\label{Lemma : Flow}
Let $\xi$ be a nowhere vanishing vector field on a closed manifold~$M$ which is tangent to a flow $\mathcal{F}$. Let $\eta$ be a~$1$-form on~$M$ which satisfies $\mathcal{L}_{\xi}\eta=0$ and $\eta(\xi)=1$. Let $h$ be a leafwise constant smooth function on $(M,\mathcal{F})$. Then there exists $f$ in $\Diff_{0}(M,\mathcal{F})$ such that
\begin{equation}\label{Equation : f and eta}
f^{*}\eta = \eta + dh.
\end{equation}
\end{lemma}

\begin{proof}
Let $\{\varphi_{s}\}_{s \in \mathbb{R}}$ be the flow generated by $\xi$. We consider maps
\begin{equation*}
\begin{array}{cccc}
\Gamma \colon & M \times \mathbb{R} & \longrightarrow & M \times \mathbb{R} \\
         &  (x,s)              & \longmapsto     & (x,h(x)s)
\end{array}
\end{equation*}
and
\begin{equation*}
\begin{array}{cccc}
\Psi \colon & M \times \mathbb{R} & \longrightarrow & M \times \mathbb{R} \\
            &  (x,s)              & \longmapsto     & (\varphi_{s}(x),s).
\end{array}
\end{equation*}
Let
\begin{equation*}
f = \pr_{1} \circ \Psi \circ \Gamma|_{M \times \{1\}},
\end{equation*}
where $\pr_{1} \colon M \times \mathbb{R} \longrightarrow M$ is the first projection. The differential map $Tf$ of $f$ at $(x,s)$ is 
\begin{equation}\label{Equation : differential}
Tf = T(\pr_{1} \circ \Psi \circ \Gamma|_{M \times \{1\}} ) = T\varphi_{1} + s (dh \otimes \xi)
\end{equation}
Then $T\varphi_{1} + s dh \otimes \xi$ is nondegenerate, because $h$ is constant on each leaf of $\mathcal{F}$. So $f$ is an open immersion. Since an open immersion from a closed manifold to a closed manifold of degree~$1$ is a diffeomorphism, $f$ is a diffeomorphism of~$M$. By~\eqref{Equation : differential}, we have
\begin{equation*}
f^{*}\eta(Y) = \eta ((T\varphi_{1} + dh \otimes \xi)(Y)) = \varphi_{1}^{*}\eta(Y) + dh(Y) = \eta(Y) + dh(Y)
\end{equation*}
for a vector field $Y$ on~$M$. Thus $f$ satisfies the equation~\eqref{Equation : f and eta}.
\end{proof}

Recall that a Sasakian metric on~$M$ is determined by a contact form $\eta$ whose Reeb flow has a transversely K\"{a}hler structure such that $d\eta$ is the transverse K\"{a}hler form (see Proposition~\ref{Proposition : Transversely Kahler}). Let $(M,\eta,{g})$ be a closed Sasakian manifold. We denote the Reeb flow of $\eta$ by $\mathcal{F}$. We denote the underlying transversely K\"{a}hler foliation by $\mathfrak{k}_{0}$. For a basic closed~$1$-form $\beta$ on $(M,\mathcal{F})$, we can construct a Sasakian metric $\sigma_{\beta}$ determined by a contact form $\eta + \beta$ and the same transversely K\"{a}hler structure of $\mathcal{F}$ as $(\eta,{g})$. We consider the set of Sasakian metrics
\begin{equation*}
\mathcal{S}_{1}=\{ \sigma_{\beta} \, | \, \beta \in \mathbf{H}^{1}_{b} \},
\end{equation*}
where $\mathbf{H}^{1}_{b}$ is the space of basic harmonic~$1$-form on $(M,\mathcal{F})$.

\begin{proposition}\label{Proposition : Intersects}
Every orbit of the action of $\Ham(\mathfrak{k}_{0})$ on $\mathcal{S}$ intersects $\mathcal{S}_{1}$.
\end{proposition}

\begin{proof}
Take a Sasakian metric $(\eta_{1},g_{1})$ whose underlying transversely K\"{a}hler flow is isomorphic to $\mathfrak{k}_{0}$. The isomorphism of transversely K\"{a}hler flows means $d\eta = d\eta_{1}$ and the transverse metrics on $\mathcal{F}$ induced by $g$ and $g_{1}$ are equal. By Lemma~\ref{Lemma : Moser}, there exist a real number $r$ and a diffeomorphism $f_{1}$ of~$M$ in $\Diff_{0}(M,\mathcal{F})$ such that
\begin{equation*}
\eta|_{T\mathcal{F}} = r f_{1}^{*}(\eta_{1}|_{T\mathcal{F}}).
\end{equation*}
Let $\eta_{2} = f_{1}^{*} \eta_{1}$. Since $(\eta - r \eta_{2})|_{T\mathcal{F}}=0$ and $d(\eta - r \eta_{2})$ is basic, $\eta - r \eta_{2}$ is basic. Hence we have
\begin{equation*}
[d\eta] = r[d\eta_{2}] = r[d\eta_{1}]
\end{equation*}
in $H^{2}_{b}(M/\mathcal{F})$. Note that $\Diff_{0}(M,\mathcal{F})$ trivially acts on the basic forms by the definition. Since we have $[d\eta_{1}] = [d\eta]$ by the assumption, we have $r=1$. Thus $\eta - \eta_{2}$ is a closed basic~$1$-form. By the basic Hodge decomposition for Riemannian foliations (see El~Kacimi~Alaoui and Hector~\cite{El Kacimi Alaoui Hector} or \'{A}lvarez L\'{o}pez~\cite{Alvarez Lopez}), there exists a basic harmonic~$1$-form $\beta$ and a smooth basic function $h_{2}$ such that
\begin{equation*}
\eta - \eta_{2} = \beta + dh_{2}.
\end{equation*}
Let $\xi$ be the common Reeb vector field of $\eta$ and $\eta_{2}$. By Lemma~\ref{Lemma : Flow}, we have a diffeomorphism $f$ of~$M$ in $\Diff_{0}(M,\mathcal{F})$ such that
\begin{equation*}
f_{2}^{*} \eta_{2} = \eta_{2} + dh_{2}.
\end{equation*}
Hence we have
\begin{equation*}
f_{2}^{*} \eta_{2} = \eta - \beta.
\end{equation*}
Thus $f_{2}^{*} \eta_{2}$ is belongs to $\mathcal{S}_{1}$. Since $f_{2}^{*} \eta_{2}$ is on the same orbit of the action of $\Diff_{0}(M,\mathcal{F})$ as $\eta_{1}$, it follows that $\mathcal{S}_{1}$ intersects with every orbit of the action of $\Diff_{0}(M,\mathcal{F})$.
\end{proof}

\begin{proposition}\label{Proposition : One point}
Every orbit of the action of $\Ham(\mathfrak{k}_{0})$ on $\mathcal{S}$ intersects $\mathcal{S}_{1}$ at most one point.
\end{proposition}

\begin{proof}
Assume that we have
\begin{equation}\label{Equation : beta 1 and 2}
\eta - \beta_{1} = f^{*}(\eta - \beta_{2})
\end{equation}
for $\beta_{1}$, $\beta_{2}$ in $\mathbf{H}^{1}_{b}$ and $f$ in $\Ham(\mathfrak{k}_{0})$. We will show $\beta_{1} = \beta_{2}$. 

Here $f^{*}\beta_{2}$ is a basic harmonic $1$-form cohomologous to $\beta_{2}$, because $f$ is isotopic to the identity and preserves the transverse metric of $\mathcal{F}$. Since each basic cohomology class is represented by a unique harmonic form by the Hodge decomposition theorem for basic cohomology (see El~Kacimi~Alaoui and Hector~\cite{El Kacimi Alaoui Hector} or \'{A}lvarez~L\'{o}pez~\cite{Alvarez Lopez}), we have 
\begin{equation}\label{Equation : beta}
f^{*}\beta_{2} = \beta_{2}.
\end{equation}

We take an isotopy $\{\varphi_{s}\}_{s \in [0,1]}$ such that $\varphi_{s}$ is belongs to $\Ham_{0}(\mathfrak{k}_{0})$ for every $s$ in $[0,1]$, and $\varphi_{0}=\id$ and $\varphi_{1}=f$. Let $X_{s} = \frac{d}{dt}\big|_{u=s}\varphi_{u}$. Since $\eta - \varphi_{s}^{*}\eta$ is exact, $\mathcal{L}_{X_{s}}\eta = \iota_{X_{s}}d\eta + d(\eta(X_{s}))$ is exact. Thus we have a smooth function $h_{s}$ on~$M$ such that
\begin{equation}\label{Equation : hamiltonian}
\iota_{X_{s}}d\eta = dh_{s}
\end{equation}
for each $s$ in $[0,1]$. Since $\mathcal{L}_{h\xi} \eta = dh$ for any smooth function $h$ on~$M$,~\eqref{Equation : hamiltonian} implies
\begin{equation}\label{Equation : alpha is preserved}
\mathcal{L}_{X_{s} - (h_{s} + \eta(X_{s}))\xi} \eta = \iota_{X_{s}}d\eta + d(\eta(X_{s})) - d(h_{s} + \eta(X_{s})) = 0.
\end{equation}
Let $\{\varphi'_{s}\}_{s \in [0,1]}$ be the isotopy generated by vector fields $\{X_{s} + (h_{s} - \eta(X_{s}))\xi\}_{s \in [0,1]}$. We have
\begin{equation}\label{Equation : alpha 3}
(\varphi'_{1})^{*}\eta = \eta
\end{equation}
by~\eqref{Equation : alpha is preserved}. Letting $f_{3} = (\varphi'_{1})^{-1} \circ f$, we have
\begin{equation}\label{Equation : f3}
f^{*}\eta = f_{3}^{*} (\varphi'_{1})^{*} \eta = f_{3}^{*} \eta
\end{equation}
by~\eqref{Equation : alpha 3}. The difference of the two vector fields which generate $\{\varphi_{s}\}_{s \in [0,1]}$ and $\{\varphi'_{s}\}_{s \in [0,1]}$ is $(h_{s} + \eta(X_{s}))\xi$, which is tangent to $\mathcal{F}$. Thus $\varphi_{s}$ and $\varphi'_{s}$ induces the same map on the leaf space of $(M,\mathcal{F})$. Thus $\{(\varphi'_{s})^{-1} \circ \varphi_{s}\}_{s \in [0,1]}$ gives an isotopy such that $(\varphi'_{0})^{-1} \circ \varphi_{0} = \id$, $(\varphi'_{1})^{-1} \circ \varphi_{1} = f_{3}$ and $(\varphi'_{s})^{-1} \circ \varphi_{s}$ is belongs to $\Diff_{0}(M,\mathcal{F})$. Let $\psi_{s} = (\varphi'_{s})^{-1} \circ \varphi_{s}$. Let $Y_{s} = \frac{d}{du}\big|_{u=s}\psi_{u}$. Since $\psi_{s}$ is belongs to $\Diff_{0}(M,\mathcal{F})$, this $Y_{s}$ is tangent to $\mathcal{F}$ for each $s$. Hence we have
\begin{equation}\label{Equation : xt}
\iota_{Y_{s}} d\eta =0.
\end{equation}
By~\eqref{Equation : xt}, we have
\begin{multline*}
f_{3}^{*}\eta - \eta = \int_{0}^{1}\frac{d}{du}\Big|_{u=s}(\psi_{u}^{*}\eta) ds = \int_{0}^{1}\psi_{s}^{*} (\mathcal{L}_{Y_{s}}\eta) ds = \\ \int_{0}^{1}\psi_{s}^{*} d(\eta(Y_{s})) ds = d\Big( \int_{0}^{1}\psi_{s}^{*} (\eta(Y_{s})) ds \Big).
\end{multline*}
\noindent Hence, letting $h_{3} =  \int_{0}^{1}\psi_{u}^{*} (\eta(Y_{s})) du$, we have
\begin{equation}\label{Equation : dh2}
f_{3}^{*}\eta - \eta = dh_{3}.
\end{equation}

By~\eqref{Equation : beta 1 and 2},~\eqref{Equation : beta},~\eqref{Equation : f3} and~\eqref{Equation : dh2}, we have
\begin{equation*}
\beta_{2} - \beta_{1} = dh_{3}. 
\end{equation*}
Since the right hand side is a harmonic form, we have $\beta_{1} = \beta_{2}$.
\end{proof}

Theorem~\ref{Theorem : H1} follows from Propositions~\ref{Proposition : Intersects} and~\ref{Proposition : One point}. 

\begin{remark}
We remark that we obtain the following result by the proof of Proposition~\ref{Proposition : Intersects}:
\begin{proposition}
Let $\eta_{1}$ and $\eta_{2}$ be the underlying contact forms on a closed manifold~$M$ of two $K$-contact structures. If $H^{1}(M;\mathbb{R})=0$ and the Reeb flows of $\eta_{1}$ and $\eta_{2}$ are equal, then there exists a real number $r$ and a diffeomorphism $f \colon M \to M$ such that $rf^{*}\eta_{2} = \eta_{1}$.
\end{proposition}
\end{remark}

\section{Examples}

\subsection{Standard Sasakian metrics on spheres}
We recall the standard Sasakian metric on $S^{2n-1}$. Consider a function $r$ on $\mathbb{R}^{2n}$ defined by
\begin{equation*}
r(x_{1}, y_{1}, x_{2}, y_{2}, \cdots, x_{n}, y_{n})=\sqrt{x_{1}^{2}+y_{1}^{2}+x_{2}^{2}+y_{2}^{2}+ \cdots +x_{n}^{2}+y_{n}^{2}},
\end{equation*}
where $(x_{1}, y_{1}, x_{2}, y_{2}, \cdots, x_{n}, y_{n})$ is the standard coordinate on $\mathbb{R}^{2n}$. Let $S^{2n-1} = r^{-1}(1)$ be the unit sphere of $\mathbb{R}^{2n}$. We consider the standard metric $g_{\std}$ and the standard contact form on $\eta_{\std}$ on $\mathbb{R}^{2n}-\{0\}$ defined by
\begin{align*}
g_{\std} & =\sum_{i=1}^{n} (dx_{i} \otimes dx_{i} + dy_{i} \otimes dy_{i}), &
\eta_{\std} & = \frac{1}{2r(x_{1},\cdots,y_{n})^{2}} \sum_{i=1}^{n} (x_{i} dy_{i} - y_{i} dx_{i}).
\end{align*}
Then $(S^{2n-1}, g_{\std}|_{S^{2n-1}}, \eta_{\std}|_{S^{2n-1}})$ is a Sasakian manifold by definition, because the K\"{a}hler manifold $(\mathbb{R}^{2n}-\{0\},g_{\std},d\eta_{\std})$ is isomorphic to $(S^{2n-1} \times \mathbb{R}_{>0}, dr^{2} + r^{2}g_{\std}|_{S^{2n-1}}, d(r^{2}\eta_{\std}|_{S^{2n-1}}))$.

The flow generated by the Reeb vector field of $\eta_{\std}|_{S^{2n-1}}$ is given by the principal $S^{1}$-action whose orbits are tangent to the fiber of Hopf fibration. The base space of the Hopf fibration is $\mathbb{C}P^{n-1}$.

To describe the deformation of transversely holomorphic flows, we use the following
\begin{proposition}[Proposition~6.1 of Girbau, Haefliger and Sundararaman~\cite{Girbau Haefliger Sundararaman}]\label{thm:GHS}
The Kuranishi space of the deformation of the transversely holomorphic flow defined by fibers of a circle bundle over a complex manifold $X$ is identified with an open neighborhood of~$0$ in $H^{0}(X,T^{1,0}X)$, the space of holomorphic vector fields on $X$, if $X$ satisfies both of the following:
\begin{align}\label{Equation : Condition 1}
H^{1,0}(X)=0, &\,\,\,  H^{1}(X,T^{1,0}X)=0.
\end{align}
\end{proposition}
\noindent Here $\mathbb{C}P^{n-1}$ satisfies $H^{1,0}(\mathbb{C}P^{n-1})=0$ and $H^{1}(\mathbb{C}P^{n-1},T^{1,0}\mathbb{C}P^{n-1})=0$. Hence Proposition~\ref{thm:GHS} implies that the Kuranishi space of deformation as a transversely holomorphic flow is identified with an open neighborhood of~$0$ in $H^{0}(\mathbb{C}P^{n-1},T^{1,0}\mathbb{C}P^{n-1})$, which is of complex dimension $n^{2}-1$. Among them, infinitesimal deformation of transversely holomorphic Riemannian flows form a union of real vector subspaces of real dimension $n-1$. 

Because $H^{0,2}(\mathbb{C}P^{n-1})=0$, Corollary~\ref{Corollary : h02 = 0} implies that the Sasakian metric $(g_{\std},\eta_{\std})$ is stable, that is, a family of compatible Sasakian metric exists for any small deformation of the Reeb flow as transversely holomorphic Riemannian flows. Since $S^{2n-1}$ is simply connected, Corollary~\ref{Corollary : H1 = 0} implies that the space of isomorphism classes of Sasakian metrics are identified with the isomorphism classes of the underlying transversely K\"{a}hler flows.

We can apply Proposition~\ref{thm:GHS}, Corollaries~\ref{Corollary : h02 = 0} and~\ref{Corollary : H1 = 0} to a Sasakian metric on the circle bundles associated to positive line bundles over Fano manifolds $X$ which satisfies $H^{1}(X,T^{1,0}X)=0$ in a similar way.

\begin{remark}
We note that deformation of the Hopf fibration as transversely holomorphic foliations is studied by Duchamp and Kalka~\cite{Duchamp Kalka} and Haefliger~\cite{Haefliger}. Other results related to Proposition~\ref{thm:GHS} is obtained by Boyer and Galicki in Section~8.2.2 of~\cite{Boyer Galicki}.
\end{remark}

\subsection{Circle bundles over complex tori}\label{Section : Example}

We present an example of deformations of the Reeb flow of a Sasakian manifold which does not admit compatible Sasakian metrics.

Let $X$ be an projective complex torus with a positive holomorphic line bundle~$E$. We fix a Hermitian metric on $E$ so that its curvature form is positive. Let~$M$ be the unit circle bundle of $E$. Then~$M$ has a Sasakian metric whose underlying transversely K\"{a}hler flow is defined by the fibers of the $S^{1}$-bundle by the Boothby-Wang construction due to Hatakeyama (Corollary of Theorem~4 of~\cite{Hatakeyama}).

It is well known that there exists a smooth family of complex tori $\{X^{t}\}_{t\in ]-1,1[}$ and a dense subset $K$ in $]-1,1[$ such that $X^{0}=X$ and $X^{t}$ is not projective for every $t$ in $K$. It is easy to see that every complex torus is K\"{a}hler. So $X^{t}$ is K\"{a}hler.

We denote the total space of the family of complex tori by $\mathcal{T}$. We fix a trivialization $\theta \colon \mathcal{T} \cong X \times ]-1,1[$ as a smooth fiber bundle over $]-1,1[$. We pull back the complex Hermitian line bundle $E$ on $X$ to $\mathcal{T}$ by $\theta \circ \pr_{1}$, where $\pr_{1} \colon X \times ]-1,1[ \to X$ is the first projection. We define $M^{t}$ as the unit circle bundle associated to the complex line bundle $(\theta^{*} \pr_{1}^{*} E)|_{X^{t}} \to X^{t}$. Let $\mathcal{F}^{t}$ be a flow on $M^{t}$ defined by the fibers of the circle bundle $M^{t} \to X^{t}$. Since the leaf space $X^{t}$ is K\"{a}hler, $\mathcal{F}^{t}$ is a transversely K\"{a}hler flow for $t$ sufficiently close to~$0$.

There exists a compatible Sasakian metric on $M^{0}$ by definition. But, for $t$ in $K$, $M^{t}$ does not have any compatible Sasakian metric. Indeed, if $M^{t}$ has a compatible Sasakian metric, then $X^{t}$ must be projective by Theorem~4 of Hatakeyama~\cite{Hatakeyama}. This is contradiction.

In this example, the basic Euler class of $\mathcal{F}^{t}$ is the Euler class of circle bundles $M^{t} \to X^{t}$ and can be considered as an element of $H^{2}(X^{t};\mathbb{Z})$. Clearly this class is of topological nature and independent of $t$. On the other hand, the Hodge decomposition $H^{2}(X^{t};\mathbb{C}) \cong H^{2,0}(X^{t}) \oplus H^{1,1}(X^{t}) \oplus H^{0,2}(X^{t})$ changes when $t$ varies.

\end{document}